\documentclass{birkjour}
\usepackage{amsmath}
\usepackage{amsxtra}
\usepackage{amscd}
\usepackage{amsthm}
\usepackage{amsfonts}
\usepackage{amssymb}
\usepackage{eucal}

\newtheorem{cor}[subsection]{Corollary}

\newtheorem{lem}[subsection]{Lemma}
\newtheorem{prop}[subsection]{Proposition}

\newtheorem{conj}[subsection]{Conjecture}
\newtheorem{thm}[subsection]{Theorem}

\newtheorem{rem}[subsection]{Remark}

\theoremstyle{definition}

\theoremstyle{remark}

\newcommand{\nc}{\newcommand}
\nc{\renc}{\renewcommand} \nc{\ssec}{\subsection}
\nc{\sssec}{\subsubsection} \nc{\on}{\operatorname}

\nc\ol{\overline} \nc\ul{\underline} \nc\wt{\widetilde}
\nc\tboxtimes{\wt{\boxtimes}} \nc{\alp}{\alpha}

\nc{\ZZ}{{\mathbb Z}} \nc{\NN}{{\mathbb N}} \nc{\CC}{{\mathbb C}}
\nc{\OO}{{\mathbb O}} \renc{\SS}{{\mathbb S}} \nc{\DD}{{\mathbb
D}}
\renewcommand{\AA}{{\mathbb A}}
\nc{\Fq}{{\mathbb F}_q} \nc{\Fqb}{\ol{{\mathbb F}_q}}
\nc{\Ql}{\ol{{\mathbb Q}_\ell}} \nc{\id}{\text{id}} \nc\X{\mathcal
X}

\nc{\Hom}{\on{Hom}} \nc{\Lie}{\on{Lie}} \nc{\Loc}{\on{Loc}}
\nc{\Pic}{\on{Pic}} \nc{\Bun}{\on{Bun}} \nc{\IC}{\on{IC}}
\nc{\Aut}{\on{Aut}} \nc{\rk}{\on{rk}} \nc{\Sh}{\on{Sh}}
\nc{\Perv}{\on{Perv}} \nc{\pos}{{\on{pos}}} \nc{\Conv}{\on{Conv}}
\nc{\Sph}{\on{Sph}} \nc{\Sym}{\on{Sym}}
\nc{\BunBb}{\overline{\Bun}_B} \nc{\Buno}{\overset{o}{\Bun}}
\nc{\BunPb}{{\overline{\Bun}_P}}
\nc{\BunBM}{\overline{\Bun}_{B(M)}}
\nc{\BunPbw}{{\widetilde{\Bun}_P}}
\nc{\BunBP}{\widetilde{\Bun}_{B,P}} \nc{\GUb}{\overline{G/U}}
\nc{\GUPb}{\overline{G/U(P)}}
\nc{\iso}{{\stackrel{\sim}{\longrightarrow}}}

\nc{\Hhom}{\underline{\on{Hom}}} \nc\syminfty{\on{Sym}^{\infty}}
\nc\lal{\ol{\lambda}} \nc\xl{\ol{x}} \nc\thl{\ol{\theta}}
\nc\nul{\ol{\nu}} \nc\mul{\ol{\mu}} \nc\Sum\Sigma
\nc{\oX}{\overset{o}{X}{}}

\nc{\M}{{\mathcal M}} \nc{\N}{{\mathcal N}} \nc{\F}{{\mathcal F}}
\nc{\D}{{\mathcal D}} \nc{\Q}{{\mathcal Q}} \nc{\Y}{{\mathcal Y}}
\nc{\G}{{\mathcal G}} \nc{\E}{{\mathcal E}} \nc{\CalC}{{\mathcal
C}}
\nc\Dh{\widehat{\D}}
\renewcommand{\O}{{\mathcal O}}
\nc{\C}{{\mathcal C}} \nc{\K}{{\mathcal K}}
\renewcommand{\H}{{\mathcal H}}

\nc{\T}{{\mathcal T}} \nc{\V}{{\mathcal V}} \renc{\P}{{\mathcal
P}} \nc{\A}{{\mathcal A}} \nc{\B}{{\mathcal B}} \nc{\U}{{\mathcal
U}}
\renewcommand{\L}{{\mathcal L}}
\nc{\Gr}{\on{Gr}}

\nc{\frn}{{\check{\mathfrak u}(P)}}
\nc\f{{\mathfrak f}}

\nc{\q}{{\mathfrak q}} \nc{\p}{{\mathfrak p}} \nc{\s}{{\mathfrak
s}} \nc\w{\text{w}}

\nc\mathi\iota \nc\Spec{\on{Spec}} \nc\Mod{\on{Mod}}
\nc{\tw}{\widetilde{\mathfrak t}} \nc{\pw}{\widetilde{\mathfrak
p}} \nc{\qw}{\widetilde{\mathfrak q}} \nc{\jw}{\widetilde j}

\nc{\grb}{\overline{\Gr}} \nc{\I}{\mathcal I}

\nc{\lambdach}{{\check\lambda}} \nc{\Lambdach}{{\check\Lambda}{}}
\nc{\much}{{\check\mu}} \nc{\omegach}{{\check\omega}}
\nc{\nuch}{{\check\nu}} \nc{\etach}{{\check\eta}}
\nc{\alphach}{{\check\alpha}} \nc{\betach}{{\check\beta}}
\nc{\rhoch}{{\check\rho}} \nc{\ch}{{\check h}}

\nc{\Hb}{\overline{\H}}

\nc{\BA}{{\mathbb{A}}} \nc{\BC}{{\mathbb{C}}}
\nc{\BM}{{\mathbb{M}}} \nc{\BN}{{\mathbb{N}}}
\nc{\BP}{{\mathbb{P}}} \nc{\BR}{{\mathbb{R}}}
\nc{\BZ}{{\mathbb{Z}}} \nc{\BS}{{\mathbb{S}}}

\nc{\CA}{{\mathcal{A}}} \nc{\CB}{{\mathcal{B}}}
\nc{\CE}{{\mathcal{E}}} \nc{\CF}{{\mathcal{F}}}
\nc{\CG}{{\mathcal{G}}} \nc{\CH}{{\mathcal{H}}}
\nc{\CI}{{\mathcal{I}}} \nc{\CL}{{\mathcal{L}}}
\nc{\CM}{{\mathcal{M}}} \nc{\CN}{{\mathcal{N}}}
\nc{\CO}{{\mathcal{O}}} \nc{\CP}{{\mathcal{P}}}
\nc{\CQ}{{\mathcal{Q}}} \nc{\CR}{{\mathcal{R}}}
\nc{\CS}{{\mathcal{S}}} \nc{\CT}{{\mathcal{T}}}
\nc{\CU}{{\mathcal{U}}} \nc{\CV}{{\mathcal{V}}}
\nc{\CW}{{\mathcal{W}}} \nc{\CZ}{{\mathcal{Z}}}

\nc{\cM}{{\check{\mathcal M}}{}} \nc{\csM}{{\check{\mathcal A}}{}}
\nc{\oM}{{\overset{\circ}{\mathcal M}}{}}
\nc{\obM}{{\overset{\circ}{\mathbf M}}{}}
\nc{\oCA}{{\overset{\circ}{\mathcal A}}{}}
\nc{\obA}{{\overset{\circ}{\mathbf A}}{}}
\nc{\ooM}{{\overset{\circ}{M}}{}}
\nc{\osM}{{\overset{\circ}{\mathsf M}}{}}
\nc{\vM}{{\overset{\bullet}{\mathcal M}}{}}
\nc{\nM}{{\underset{\bullet}{\mathcal M}}{}}
\nc{\oD}{{\overset{\circ}{\mathcal D}}{}}
\nc{\obD}{{\overset{\circ}{\mathbf D}}{}}
\nc{\oA}{{\overset{\circ}{\mathbb A}}{}}
\nc{\op}{{\overset{\bullet}{\mathbf p}}{}}
\nc{\cp}{{\overset{\circ}{\mathbf p}}{}}
\nc{\oU}{{\overset{\bullet}{\mathcal U}}{}}
\nc{\oZ}{{\overset{\circ}{\mathcal Z}}{}}
\nc{\ofZ}{{\overset{\circ}{\mathfrak Z}}{}}

\nc{\ff}{{\mathfrak{f}}} \nc{\fv}{{\mathfrak{v}}}
\nc{\fa}{{\mathfrak{a}}} \nc{\fb}{{\mathfrak{b}}}
\nc{\fd}{{\mathfrak{d}}} \nc{\fe}{{\mathfrak{e}}}
\nc{\fg}{{\mathfrak{g}}} \nc{\fgl}{{\mathfrak{gl}}}
\nc{\fh}{{\mathfrak{h}}} \nc{\fri}{{\mathfrak{i}}}
\nc{\fj}{{\mathfrak{j}}} \nc{\fk}{{\mathfrak{k}}}
\nc{\fm}{{\mathfrak{m}}} \nc{\fn}{{\mathfrak{n}}}
\nc{\ft}{{\mathfrak{t}}} \nc{\fu}{{\mathfrak{u}}}
\nc{\fw}{{\mathfrak{w}}} \nc{\fz}{{\mathfrak{z}}}
\nc{\fp}{{\mathfrak{p}}} \nc{\frr}{{\mathfrak{r}}}
\nc{\fs}{{\mathfrak{s}}} \nc{\fsl}{{\mathfrak{sl}}}
\nc{\hsl}{{\widehat{\mathfrak{sl}}}}
\nc{\hgl}{{\widehat{\mathfrak{gl}}}}
\nc{\hg}{{\widehat{\mathfrak{g}}}}
\nc{\chg}{{\widehat{\mathfrak{g}}}{}^\vee}
\nc{\hn}{{\widehat{\mathfrak{n}}}}
\nc{\chn}{{\widehat{\mathfrak{n}}}{}^\vee}

\nc{\fA}{{\mathfrak{A}}} \nc{\fB}{{\mathfrak{B}}}
\nc{\fD}{{\mathfrak{D}}} \nc{\fE}{{\mathfrak{E}}}
\nc{\fF}{{\mathfrak{F}}} \nc{\fG}{{\mathfrak{G}}}
\nc{\fI}{{\mathfrak{I}}} \nc{\fJ}{{\mathfrak{J}}}
\nc{\fK}{{\mathfrak{K}}} \nc{\fL}{{\mathfrak{L}}}
\nc{\fM}{{\mathfrak{M}}} \nc{\fN}{{\mathfrak{N}}}
\nc{\frP}{{\mathfrak{P}}} \nc{\fQ}{{\mathfrak{Q}}}
\nc{\fT}{{\mathfrak{T}}} \nc{\fU}{{\mathfrak{U}}}
\nc{\fV}{{\mathfrak{V}}} \nc{\fW}{{\mathfrak{W}}}
\nc{\fX}{{\mathfrak{X}}} \nc{\fY}{{\mathfrak{Y}}}
\nc{\fZ}{{\mathfrak{Z}}}

\nc{\bb}{{\mathbf{b}}} \nc{\bc}{{\mathbf{c}}}
\nc{\be}{{\mathbf{e}}} \nc{\bj}{{\mathbf{j}}}
\nc{\bn}{{\mathbf{n}}} \nc{\bp}{{\mathbf{p}}}
\nc{\bq}{{\mathbf{q}}} \nc{\bfr}{{\mathbf{r}}}
\nc{\bfu}{{\mathbf{u}}} \nc{\bv}{{\mathbf{v}}}
\nc{\bx}{{\mathbf{x}}} \nc{\by}{{\mathbf{y}}}
\nc{\bw}{{\mathbf{w}}} \nc{\bA}{{\mathbf{A}}}
\nc{\bB}{{\mathbf{B}}} \nc{\bC}{{\mathbf{C}}}
\nc{\bD}{{\mathbf{D}}} \nc{\bF}{{\mathbf{F}}}
\nc{\bH}{{\mathbf{H}}} \nc{\bK}{{\mathbf{K}}}
\nc{\bM}{{\mathbf{M}}} \nc{\bN}{{\mathbf{N}}}
\nc{\bO}{{\mathbf{O}}} \nc{\bS}{{\mathbf{S}}}
\nc{\bV}{{\mathbf{V}}} \nc{\bW}{{\mathbf{W}}}
\nc{\bX}{{\mathbf{X}}} \nc{\bP}{{\mathbf{P}}}
\nc{\bZ}{{\mathbf{Z}}}

\nc{\sA}{{\mathsf{A}}} \nc{\sB}{{\mathsf{B}}}
\nc{\sC}{{\mathsf{C}}} \nc{\sD}{{\mathsf{D}}}
\nc{\sE}{{\mathsf{E}}} \nc{\sF}{{\mathsf{F}}}
\nc{\sK}{{\mathsf{K}}} \nc{\sL}{{\mathsf{L}}}
\nc{\sM}{{\mathsf{M}}} \nc{\sO}{{\mathsf{O}}}
\nc{\sQ}{{\mathsf{Q}}} \nc{\sP}{{\mathsf{P}}}
\nc{\sT}{{\mathsf{T}}} \nc{\sZ}{{\mathsf{Z}}}
\nc{\sV}{{\mathsf{V}}} \nc{\sU}{{\mathsf{U}}}
\nc{\sfp}{{\mathsf{p}}} \nc{\sr}{{\mathsf{r}}}
\nc{\st}{{\mathsf{t}}} \nc{\sfb}{{\mathsf{b}}}
\nc{\sfc}{{\mathsf{c}}} \nc{\sd}{{\mathsf{d}}}
\nc{\sz}{{\mathsf{z}}}

\nc{\BK}{{\bar{K}}}

\nc{\tA}{{\widetilde{\mathbf{A}}}}
\nc{\tB}{{\widetilde{\mathcal{B}}}}
\nc{\tg}{{\widetilde{\mathfrak{g}}}} \nc{\tG}{{\widetilde{G}}}
\nc{\TM}{{\widetilde{\mathbb{M}}}{}}
\nc{\tO}{{\widetilde{\mathsf{O}}}{}}
\nc{\tU}{{\widetilde{\mathfrak{U}}}{}} \nc{\TZ}{{\tilde{Z}}}
\nc{\tx}{{\tilde{x}}} \nc{\tbv}{{\tilde{\bv}}}
\nc{\tfP}{{\widetilde{\mathfrak{P}}}{}} \nc{\tz}{{\tilde{\zeta}}}
\nc{\tmu}{{\tilde{\mu}}}

\nc{\urho}{\underline{\rho}} \nc{\uB}{\underline{B}}
\nc{\uC}{{\underline{\mathbb{C}}}} \nc{\ui}{\underline{i}}
\nc{\uj}{\underline{j}} \nc{\ofP}{{\overline{\mathfrak{P}}}}
\nc{\oB}{{\overline{\mathcal{B}}}}
\nc{\og}{{\overline{\mathfrak{g}}}} \nc{\oI}{{\overline{I}}}

\nc{\eps}{\varepsilon} \nc{\hrho}{{\hat{\rho}}}

\nc{\one}{{\mathbf{1}}} \nc{\two}{{\mathbf{t}}}

\nc{\Rep}{{\mathop{\operatorname{\rm Rep}}}}
\nc{\Tot}{{\mathop{\operatorname{\rm Tot}}}}
\nc{\Ker}{{\mathop{\operatorname{\rm Ker}}}}
\nc{\Hilb}{{\mathop{\operatorname{\rm Hilb}}}}
\nc{\End}{{\mathop{\operatorname{\rm End}}}}
\nc{\Ext}{{\mathop{\operatorname{\rm Ext}}}}
\nc{\CHom}{{\mathop{\operatorname{{\mathcal{H}}\it om}}}}
\nc{\GL}{{\mathop{\operatorname{\rm GL}}}}
\nc{\gr}{{\mathop{\operatorname{\rm gr}}}}
\nc{\Id}{{\mathop{\operatorname{\rm Id}}}}
\nc{\defi}{{\mathop{\operatorname{\rm def}}}}
\nc{\length}{{\mathop{\operatorname{\rm length}}}}
\nc{\supp}{{\mathop{\operatorname{\rm supp}}}}

\nc{\Cliff}{{\mathsf{Cliff}}}
\nc{\Fl}{{\mathsf{Fl}}} \nc{\Fib}{{\mathsf{Fib}}}
\nc{\Coh}{{\mathsf{Coh}}} \nc{\FCoh}{{\mathsf{FCoh}}}

\nc{\reg}{{\text{\rm reg}}}

\nc{\cplus}{{\mathbf{C}_+}} \nc{\cminus}{{\mathbf{C}_-}}
\nc{\cthree}{{\mathbf{C}_*}} \nc{\Qbar}{{\bar{Q}}}

\nc{\bh}{{\bar{h}}} \nc{\bOmega}{{\overline{\Omega}}}

\nc{\seq}[1]{\stackrel{#1}{\sim}}

\nc{\aff}{\operatorname{aff}}

\renewcommand{\gg}{{\mathfrak{g}}}
\newcommand{\hh}{{\mathfrak{h}}}

\dedicatory{To the memory of Izrail Moiseevich Gelfand}
\begin{document}

\author{Boris Feigin, Michael Finkelberg, Igor Frenkel and Leonid Rybnikov}
\title
{Gelfand-Tsetlin algebras and cohomology rings of Laumon spaces}





\address{{\it Address}:\newline
B.F.: Landau Institute for Theoretical Physics, Kosygina st 2,
Moscow 117940, Russia \newline
M.F.: IMU, IITP, and
State University Higher School of Economics, \newline
Department of Mathematics, \newline
20 Myasnitskaya st,
Moscow 101000, Russia \newline
I.F.: Yale University, Department of Mathematics,
PO Box 208283, New Haven, CT 06520, USA \newline
L.R.: Institute for the Information Transmission Problems and \newline
State University Higher School of Economics, \newline
Department of Mathematics, \newline
20 Myasnitskaya st, Moscow 101000, Russia}

\email{\newline bfeigin@gmail.com, fnklberg@gmail.com,
frenkel-igor@yale.edu, leo.rybnikov@gmail.com}

\begin{abstract}
Laumon moduli spaces are certain smooth closures of the moduli spaces of
maps from the projective line to the flag variety of $GL_n$.
We calculate the equvariant cohomology rings of the Laumon moduli spaces
in terms of Gelfand-Tsetlin subalgebra of $U(gl_n)$,
and formulate a conjectural answer for the small quantum cohomology rings in
terms of certain commutative shift of argument subalgebras of $U(gl_n)$.
\end{abstract}
\maketitle

\section{Introduction}
\subsection{Cohomology of Laumon spaces}
The moduli spaces $\CQ_{\ul{d}}$ were introduced by G.~Laumon in~\cite{la1}
and~\cite{la2}. They are certain compactifications
of the moduli spaces of degree
$\ul{d}$ maps from $\BP^1$ to the flag variety $\CB_n$ of $GL_n$. The original
motivation of G.~Laumon was to study the geometric Eisenstein series, but
later the Laumon moduli spaces proved useful also in the computation of quantum
cohomology and $K$-theory of $\CB_n$, see e.g.~\cite{gl},~\cite{bf}. The aim of
the present note is to calculate the cohomology rings of the
Laumon moduli spaces,
and to formulate a conjectural answer for the quantum cohomology rings.

The main tool is the action of the universal enveloping algebra $U(\fgl_n)$
by correspondences~\cite{fk} on the direct sum (over all degrees)
of cohomology of $\CQ_{\ul{d}}$. More precisely, we consider the localized
equivariant cohomology $B:=\bigoplus_{\ul{d}}H^\bullet_{GL_n\times\BC^*}
(\CQ_{\ul{d}})\otimes_{H^\bullet_{GL_n\times\BC^*}(pt)}\on{Frac}
(H^\bullet_{GL_n\times\BC^*}(pt))$ where $\BC^*$ acts as ``loop rotations'' on
the source $\BP^1$, while $GL_n$ acts naturally on the target $\CB_n$.
We also consider a ``local version'' $\fQ_{\ul{d}}$ of the Laumon moduli
space, which is a certain closure of the moduli space of {\em based} maps
of degree $\ul{d}$ from $\BP^1$ to $\CB_n$. This local version does not
carry the action of the whole group $GL_n\times\BC^*$, but only of the Cartan
torus $\widetilde{T}\times\BC^*$. Accordingly, we consider the equivariant
cohomology (resp. localized equivariant cohomology)
$'V=\bigoplus_{\ul{d}}H^\bullet_{\widetilde{T}\times\BC^*}(\fQ_{\ul{d}})$
(resp. $V=\bigoplus_{\ul{d}}H^\bullet_{\widetilde{T}\times\BC^*}(\fQ_{\ul{d}})
\otimes_{H^\bullet_{\widetilde{T}\times\BC^*}(pt)}\on{Frac}
(H^\bullet_{\widetilde{T}\times\BC^*}(pt))$).

According to~\cite{b} (cf. also~\cite{ne} and our Theorem~\ref{brav}),
the above action of $U(\fgl_n)$ identifies $V$ with
the universal Verma module $\fV$. Similarly, $B$ carries the action of
{\em two copies} of $U(\fgl_n)$ by correspondences, and can be identified
with the tensor square $\fB$ of $\fV$ (Theorem~\ref{prav}).
The nonlocalized cohomology $'V$ is identified with a certain
integral form $\ul{\fV}$ of $\fV$, a version of the universal {\em dual} Verma
module (Theorem~\ref{braverman}). We were unable to describe the {\em
nonlocalized} equivariant cohomology of $\bigsqcup_{\ul{d}}\CQ_{\ul{d}}$
as a $U(\fgl_n)^2$-module, but we propose a conjecture~\ref{ancient} in this
direction; it is an equivariant generalization of Conjecture~6.4 of~\cite{fk}.

\subsection{Gelfand-Tsetlin algebra}
The description of the cohomology rings is given in representation theoretic
terms. Namely, the universal enveloping algebra of $\fgl_n$ contains the
{\em Gelfand-Tsetlin subalgebra} $\ul{\fG}$ (a maximal commutative subalgebra).
For a given degree $\ul{d}$, the equivariant cohomology $'V_{\ul{d}}$ is
identified with the weight subspace $\ul{\fV}_{\ul{d}}$. The identity element
$1_{\ul{d}}$ of the cohomology ring goes to the weight component $\fv_{\ul{d}}$
of the Whittaker vector $\fv\in\ul{\fV}$. It turns out that the vector
$1_{\ul{d}}\in\ 'V$ is cyclic for $\ul{\fG}$; hence the equivariant
cohomology ring $H^\bullet_{\widetilde{T}\times\BC^*}(\fQ_{\ul{d}})$ is
identified with a quotient of the Gelfand-Tsetlin subalgebra
(Corollary~\ref{mor}).
Similar results hold for the localized equivariant cohomology of
$\fQ_{\ul{d}}$ and $\CQ_{\ul{d}}$ (Propositions~\ref{gc} and~\ref{any}).

The proof uses two ingredients. First, the localized equivariant cohomology
has a natural basis of classes of the torus fixed points. We check that under
the identification $V\simeq\fV$, this basis goes to the Gelfand-Tsetlin basis
of $\fV$. Also, the cohomology of $\fQ_{\ul{d}}$ contain the
(K\"unneth components of the) Chern classes of the universal tautological
vector bundles on $\fQ_{\ul{d}}\times\BP^1$. The operators of multiplication
by these Chern classes are diagonal in the fixed point basis, and by comparison
to the Gelfand-Tsetlin basis it is possible to identify these operators with
the action of certain generators of $\ul{\fG}$. Finally, since the diagonal
class of $\fQ_{\ul{d}}$ is decomposable,
the above Chern classes generate the cohomology ring of $\fQ_{\ul{d}}$.

In a similar vein, in Proposition~\ref{GC} we compute the localized
equivariant $K$-ring of $\fQ_{\ul{d}}$ in terms of the
``quantum Gelfand-Tsetlin algebra''.

\subsection{Quantum cohomology of Laumon spaces}
The Picard group of the local Laumon space $\fQ_{\ul{d}}$ is free of rank
$n-2$ iff all the entries of $\ul{d}$ are nonzero. It possesses the set
of distinguished generators: the classes of determinant line bundles
$\D_2,\ldots,\D_{n-1}$. Let $\mathsf T$ be a torus with the cocharacter
lattice $\on{Pic}(\fQ_{\ul{d}})$, and let $q_i,\ 2\leq i\leq n-1$, be
the coordinates on $\mathsf T$ corresponding to $\D_i$. We conjecture a
formula for the operator $M_{\D_i}$ of {\em quantum} multiplication by the
first Chern class $c_1(\D_i)$. {\em A priori} this operator lies in
$\End(V_{\ul{d}})[[q_2,\ldots,q_{n-1}]]$, but according to
Conjecture~\ref{Lenya},
it is the Taylor expansion of a rational $\End(V_{\ul{d}})$-valued function
on $\mathsf T$. Moreover, this function arises from the action of a
{\em universal} element $QC_i\in U(\fgl_n)$
(depending on $q_2,\ldots,q_{n-1}$) on the weight space $V_{\ul{d}}$ of the
universal Verma module. The commutant of the collection of all such elements
$\{QC_i(q_2,\ldots,q_{n-1})\}$ is a {\em shift of argument subalgebra}
$\A_q\subset U(\fgl_n)$ (a maximal commutative subalgebra, see~\cite{Ryb1}).

We consider the flat $\End(V_{\ul{d}})$-valued connection on $\mathsf{T}:\
\nabla=\sum_{i=2}^{n-1}q_i\frac{\partial}{\partial q_i}+QC_i$ (the {\em
quantum connection}).
Conjecture~\ref{Lenya} (recently proved by A.~Negut) implies that $\nabla$ is
induced by the {\em Casimir connection}~\cite{dc,fmtv,TL,MTL}
on the Cartan subalgebra
$\fh\subset\fsl_n$ under an embedding $\mathsf{T}\hookrightarrow\fh$.
In particular, $\nabla$ has regular singularities, and its monodromy
factors through the action of the pure braid group $PB_n$
(fundamental group of the complement in $\fh$ to the root hyperplanes)
on the weight space $V_{\ul{d}}$ by the ``quantum Weyl group operators".

\subsection{Acknowledgments}
This note is a result of discussions with many people. In particular, the
``restricted'' correspondences $\ul{\fE}{}_{\ul{d},\alpha_{ij}},
\CE^0_{\ul{d},\alpha_{ij}}, \CE^\infty_{\ul{d},\alpha_{ij}}$
(see~\ref{kuzn},~\ref{if}) and the action of $U(\fgl_n)\ltimes\BC[\fgl_n]$
on nonequivariant cohomology of $\bigsqcup_{\ul{d}}\CQ_{\ul{d}}$
were introduced by A.~Kuznetsov in 1998. The idea to consider the
action of correspondences on the equivariant cohomology of Laumon spaces
was proposed by V.~Schechtman in 1997.
I.~F. would like to thank A.~Licata and A.~Marian
for many useful discussions of
representation theory based on Laumon's spaces.
In fact, some of the calculations associated with the double
$\fgl_n$ action of
Section~\ref{GLS} were carried out independently by A.~Licata and A.~Marian,
also in a nonequivariant setting. Furthermore, we learnt from A.~Marian
that the Chern classes of the tautological bundles generate the
cohomology ring of $\fQ_{\ul{d}}$.
A.~I.~Molev provided us with references on Gelfand-Tsetlin bases.
Above all, our interest in quantum
cohomology of Laumon spaces was inspired by R.~Bezrukavnikov,
A.~Braverman, A.~Negut,
A.~Okounkov and V.~Toledano Laredo during the beautiful special year
2007/2008 at IAS organized by R.~Bezrukavnikov. We are deeply grateful
to all of them. Finally, we are obliged to A.~Tsymbaliuk for the careful
reading of the first version of our note and spotting several mistakes.
B.~F. was partially supported by the grants
RFBR 05-01-01007, RFBR 05-01-01934, and NSh-6358.2006.2.
I.~F. was supported by the NSF grant DMS-0457444.
M.~F. was partially supported by the Oswald Veblen Fund,
RFBR grant 09-01-00242, the Science Foundation of the
SU-HSE awards No.T3-62.0 and 10-09-0015, and the Ministry of Education and
Science of Russian Federation, grant No. 2010-1.3.1-111-017-029.
The work of L.~R. was partially supported
by  RFBR grants 07-01-92214-CNRSL-a, 05-01-02805-CNRSL-a, 09-01-00242,
the Science Foundation of the SU-HSE awards No.T3-62.0 and 10-09-0015, and the Ministry of Education and
Science of Russian Federation, grant No. 2010-1.3.1-111-017-029.
He gratefully acknowledges the support from Deligne 2004 Balzan prize in mathematics.

\section{Local Laumon spaces}

\subsection{Laumon spaces}
We recall the setup of ~\cite{fk},~\cite{bf}. Let $\bC$ be a smooth
projective curve of genus zero. We fix a coordinate $z$ on $\bC$,
and consider the action of $\BC^*$ on $\bC$ such that
$v(z)=v^{-2}z$. We have $\bC^{\BC^*}=\{0,\infty\}$.

We consider an $n$-dimensional vector space $W$ with a basis
$w_1,\ldots,w_n$. This defines a Cartan torus $T\subset
G=GL_n\subset Aut(W)$. We also consider its $2^n$-fold cover,
the bigger torus $\widetilde{T}$, acting on $W$ as follows: for
$\widetilde{T}\ni\ul{t}=(t_1,\ldots,t_n)$ we have
$\ul{t}(w_i)=t_i^2w_i$. We denote by $\CB$ the flag variety of
$G$.

Given an $(n-1)$-tuple of nonnegative integers
$\ul{d}=(d_1,\ldots,d_{n-1})$, we consider the Laumon's
quasiflags' space $\CQ_{\ul{d}}$, see ~\cite{la2}, ~4.2. It is the
moduli space of flags of locally free subsheaves
$$0\subset\CW_1\subset\ldots\subset\CW_{n-1}\subset\CW=W\otimes\CO_\bC$$
such that $\on{rank}(\CW_k)=k$, and $\deg(\CW_k)=-d_k$.

It is known to be a smooth projective variety of dimension
$2d_1+\ldots+2d_{n-1}+\dim\CB$, see ~\cite{la1}, ~2.10.

We consider the following locally closed subvariety
$\fQ_{\ul{d}}\subset\CQ_{\ul{d}}$ (quasiflags based at
$\infty\in\bC$) formed by the flags
$$0\subset\CW_1\subset\ldots\subset\CW_{n-1}\subset\CW=W\otimes\CO_\bC$$
such that $\CW_i\subset\CW$ is a vector subbundle in a
neighbourhood of $\infty\in\bC$, and the fiber of $\CW_i$ at
$\infty$ equals the span $\langle w_1,\ldots,w_i\rangle\subset W$.

It is known to be a smooth quasiprojective variety of dimension
$2d_1+\ldots+2d_{n-1}$.

\subsection{Fixed points}
\label{fixed points} The group $G\times\BC^*$ acts naturally on
$\CQ_{\ul{d}}$, and the group $\widetilde{T}\times\BC^*$ acts
naturally on $\fQ_{\ul{d}}$. The set of fixed points of
$\widetilde{T}\times\BC^*$ on $\fQ_{\ul{d}}$ is finite; we recall
its description from ~\cite{fk}, ~2.11.

Let $\widetilde{\ul{d}}$ be a collection of nonnegative integers
$(d_{ij}),\ i\geq j$, such that $d_i=\sum_{j=1}^id_{ij}$, and for
$i\geq k\geq j$ we have $d_{kj}\geq d_{ij}$. Abusing notation we
denote by $\widetilde{\ul{d}}$ the corresponding
$\widetilde{T}\times\BC^*$-fixed point in $\fQ_{\ul{d}}$:

$\CW_1=\CO_\bC(-d_{11}\cdot0)w_1,$

$\CW_2=\CO_\bC(-d_{21}\cdot0)w_1\oplus\CO_\bC(-d_{22}\cdot0)w_2,$

$\ldots\ \ldots\ \ldots\ ,$

$\CW_{n-1}=\CO_\bC(-d_{n-1,1}\cdot0)w_1\oplus\CO_\bC(-d_{n-1,2}\cdot0)w_2
\oplus\ldots\oplus\CO_\bC(-d_{n-1,n-1}\cdot0)w_{n-1}.$

\subsection{Correspondences}
\label{classic}
For $i\in\{1,\ldots,n-1\}$, and $\ul{d}=(d_1,\ldots,d_{n-1})$, we
set $\ul{d}+i:=(d_1,\ldots,d_i+1,\ldots,d_{n-1})$. We have a
correspondence $\CE_{\ul{d},i}\subset\CQ_{\ul{d}}\times
\CQ_{\ul{d}+i}$ formed by the pairs $(\CW_\bullet,\CW'_\bullet)$
such that for $j\ne i$ we have $\CW_j=\CW'_j$, and
$\CW'_i\subset\CW_i$, see ~\cite{fk}, ~3.1. In other words,
$\CE_{\ul{d},i}$ is the moduli space of flags of locally free
sheaves
$$0\subset\CW_1\subset\ldots\CW_{i-1}\subset\CW'_i\subset\CW_i\subset
\CW_{i+1}\ldots\subset\CW_{n-1}\subset\CW$$ such that
$\on{rank}(\CW_k)=k$, and $\deg(\CW_k)=-d_k$, while
$\on{rank}(\CW'_i)=i$, and $\deg(\CW'_i)=-d_i-1$.

According to ~\cite{la1}, ~2.10, $\CE_{\ul{d},i}$ is a smooth
projective algebraic variety of dimension
$2d_1+\ldots+2d_{n-1}+\dim\CB+1$.

We denote by $\bp$ (resp. $\bq$) the natural projection
$\CE_{\ul{d},i}\to\CQ_{\ul{d}}$ (resp.
$\CE_{\ul{d},i}\to\CQ_{\ul{d}+i}$). We also have a map $\bfr:\
\CE_{\ul{d},i}\to\bC,$
$$(0\subset\CW_1\subset\ldots\CW_{i-1}\subset\CW'_i\subset\CW_i\subset
\CW_{i+1}\ldots\subset\CW_{n-1}\subset\CW)\mapsto\on{supp}(\CW_i/\CW'_i).$$

The correspondence $\CE_{\ul{d},i}$ comes equipped with a natural
line bundle $\CL_i$ whose fiber at a point
$$(0\subset\CW_1\subset\ldots\CW_{i-1}\subset\CW'_i\subset\CW_i\subset
\CW_{i+1}\ldots\subset\CW_{n-1}\subset\CW)$$ equals
$\Gamma(\bC,\CW_i/\CW'_i)$.

Finally, we have a transposed correspondence
$^\sT\CE_{\ul{d},i}\subset \CQ_{\ul{d}+i}\times\CQ_{\ul{d}}$.

Restricting to $\fQ_{\ul{d}}\subset\CQ_{\ul{d}}$ we obtain the
correspondence
$\fE_{\ul{d},i}\subset\fQ_{\ul{d}}\times\fQ_{\ul{d}+i}$ together
with line bundle $\fL_i$ and the natural maps $\bp:\
\fE_{\ul{d},i}\to\fQ_{\ul{d}},\ \bq:\
\fE_{\ul{d},i}\to\fQ_{\ul{d}+i},\ \bfr:\
\fE_{\ul{d},i}\to\bC-\infty$. We also have a transposed
correspondence $^\sT\fE_{\ul{d},i}\subset
\fQ_{\ul{d}+i}\times\fQ_{\ul{d}}$. It is a smooth quasiprojective
variety of dimension $2d_1+\ldots+2d_{n-1}+1$.

\subsection{Equivariant cohomology}
We denote by ${}'V$ the direct sum of equivariant (complexified)
cohomology:
${}'V=\oplus_{\ul{d}}H^\bullet_{\widetilde{T}\times\BC^*}(\fQ_{\ul{d}})$.
It is a module over
$H^\bullet_{\widetilde{T}\times\BC^*}(pt)=\BC[\ft\oplus\BC]=
\BC[x_1,\ldots,x_n,\hbar]$. Here $\ft\oplus\BC$ is the
Lie algebra of $\widetilde{T}\times\BC^*$. We define $\hbar$ as twice
the positive generator of $H^2_{\BC^*}(pt,\BZ)$. Similarly, we define
$x_i\in H^2_{\widetilde{T}}(pt,\BZ)$ in terms of the corresponding
one-parametric subgroup.
We define $V=\ {}'V\otimes_{H^\bullet_{\widetilde{T}\times\BC^*}(pt)}
\on{Frac}(H^\bullet_{\widetilde{T}\times\BC^*}(pt))$.

We have an evident grading $V=\oplus_{\ul{d}}V_{\ul{d}},\
V_{\ul{d}}=H^\bullet_{\widetilde{T}\times\BC^*}(\fQ_{\ul{d}})
\otimes_{H^\bullet_{\widetilde{T}\times\BC^*}(pt)}
\on{Frac}(H^\bullet_{\widetilde{T}\times\BC^*}(pt))$.

\subsection{Universal Verma module}
We denote by
$\fU$ the universal enveloping algebra of $\fgl_n$ over the field
$\BC(\ft\oplus\BC)$. For $1\leq j,k\leq n$ we denote by
$E_{jk}\in\fgl_n\subset\fU$ the usual elementary matrix.
The standard Chevalley generators are expressed as follows:
$$\fe_i:=E_{i+1,i},\
\ff_i:=E_{i,i+1},\ \fh_i:=E_{i+1,i+1}-E_{ii}$$
(note that $\fe_i$ is represented by a {\em lower} triangular matrix).
Note also that $\fU$ is generated by $E_{ii},\ 1\leq i\leq n,\
E_{i,i+1},E_{i+1,i},\ 1\leq i\leq n-1$.
We denote by $\fU_{\leq0}$ the subalgebra of $\fU$ generated by
$E_{ii},\ 1\leq i\leq n,\ E_{i,i+1},\ 1\leq i\leq n-1$.
It acts on the field $\BC(\ft\oplus\BC)$ as follows:
$E_{i,i+1}$ acts trivially for any $1\leq i\leq n-1$, and
$E_{ii}$ acts by multiplication by
$\hbar^{-1}x_i+i-1$. We define the {\em universal Verma module}
$\fV$ over $\fU$ as $\fU\otimes_{\fU_{\leq0}}\BC(\ft\oplus\BC)$.
The universal Verma module $\fV$ is an irreducible $\fU$-module.

\subsection{The action of generators}
\label{operators} The grading and the correspondences
$^\sT\fE_{\ul{d},i},\fE_{\ul{d},i}$ give rise to the following
operators on $V$ (note that though $\bp$ is not proper, $\bp_*$ is
well defined on the localized equivariant cohomology due to the
finiteness of the fixed point sets and smoothness of $\fE_{\ul{d},i}$):

$E_{ii}=\hbar^{-1}x_i+d_{i-1}-d_i+i-1:\ V_{\ul{d}}\to
V_{\ul{d}}$;

$\fh_i=\hbar^{-1}(x_{i+1}-x_i)+2d_i-d_{i-1}-d_{i+1}+1:\ V_{\ul{d}}\to
V_{\ul{d}}$;

$\ff_i=E_{i,i+1}=\bp_*\bq^*:\ V_{\ul{d}}\to V_{\ul{d}-i}$;

$\fe_i=E_{i+1,i}=-\bq_*\bp^*:\
V_{\ul{d}}\to V_{\ul{d}+i}$.

\begin{thm}
\label{brav}
The operators $\fe_i=E_{i+1,i},E_{ii},\ff_i=E_{i,i+1}$ on $V$ defined
in ~\ref{operators} satisfy the relations in $\fU$, i.e.
they give rise to the action of $\fU$ on $V$.
There is a unique isomorphism $\Psi$ of $\fU$-modules
$V$ and $\fV$ carrying $1\in H^0_{\widetilde{T}\times\BC^*}(\fQ_0)\subset V$
to the lowest weight vector $1\in\BC(\ft\oplus\BC)\subset\fV$.
\end{thm}

The proof is entirely similar to the proof of Theorem~2.12 of~\cite{bf};
cf. also~\cite{ne}. \qed

\subsection{Fixed point basis}
\label{classical}

According to the Localization theorem in equivariant cohomology (see e.g.
~\cite{cg}), restriction to the $\widetilde{T}\times\BC^*$-fixed
point set induces an isomorphism
$$H^\bullet_{\widetilde{T}\times\BC^*}(\fQ_{\ul{d}})
\otimes_{H^\bullet_{\widetilde{T}\times\BC^*}(pt)}
\on{Frac}(H^\bullet_{\widetilde{T}\times\BC^*}(pt))\to
H^\bullet_{\widetilde{T}\times\BC^*}(\fQ_{\ul{d}}^{\widetilde{T}\times\BC^*})
\otimes_{H^\bullet_{\widetilde{T}\times\BC^*}(pt)}
\on{Frac}(H^\bullet_{\widetilde{T}\times\BC^*}(pt))$$

The fundamental cycles $[\widetilde{\ul{d}}]$ of the
$\widetilde{T}\times\BC^*$-fixed points $\widetilde{\ul{d}}$ (see
~\ref{fixed points}) form a basis in
$\oplus_{\ul{d}}H^\bullet_{\widetilde{T}\times\BC^*}
(\fQ_{\ul{d}}^{\widetilde{T}\times\BC^*})
\otimes_{H^\bullet_{\widetilde{T}\times\BC^*}(pt)}\on{Frac}
(H^\bullet_{\widetilde{T}\times\BC^*}(pt))$. The embedding of a point
$\widetilde{\ul{d}}$ into $\fQ_{\ul{d}}$ is a proper morphism, so the
direct image in the equivariant cohomology is well defined, and we will
denote by $[\widetilde{\ul{d}}]\in V_{\ul{d}}$ the direct image of the
fundamental cycle of the point $\widetilde{\ul{d}}$. The set
$\{[\widetilde{\ul{d}}]\}$ forms a basis of $V$.

The matrix coefficients of the operators $\fe_i,\ff_i$ in the basis
$\{[\widetilde{\ul{d}}]\}$
were computed in~\cite{bf}; cf. also~\cite{ne}~8.2. The result is:

\begin{prop}
\label{feigin'}
The matrix coefficients of the operators $\fe_i,\ff_i$ in the basis
$\{[\widetilde{\ul{d}}]\}$ are as follows:
$$\fe_{i[\widetilde{\ul{d}},\widetilde{\ul{d}}{}']}=
-\hbar^{-1}\prod_{j\ne k\leq i}(x_j-x_k+(d_{i,k}-d_{i,j})\hbar)^{-1}
\prod_{k\leq i-1}(x_j-x_k+(d_{i-1,k}-d_{i,j})\hbar)$$ if
$d'_{i,j}=d_{i,j}+1$ for certain $j\leq i$;

$$\ff_{i[\widetilde{\ul{d}},\widetilde{\ul{d}}{}']}=
\hbar^{-1}\prod_{j\ne k\leq i}(x_k-x_j+(d_{i,j}-d_{i,k})\hbar)^{-1}
\prod_{k\leq i+1}(x_k-x_j+(d_{i,j}-d_{i+1,k})\hbar)$$ if
$d'_{i,j}=d_{i,j}-1$ for certain $j\leq i$;

All the other matrix coefficients of $\fe_i,\ff_i$ vanish.
\end{prop}

The proof is entirely similar to that of Corollary~2.20 of~\cite{bf}.

\subsection{Gelfand-Tsetlin basis of the universal Verma module}
\label{classical'}
We will follow the notations of~\cite{m} on the Gelfand-Tsetlin bases in
representations of $\fgl_n$.
To a collection $\widetilde{\ul{d}}=(d_{ij}),\ n-1\geq i\geq j$ we associate
a {\em Gelfand-Tsetlin pattern}
$\Lambda=\Lambda(\widetilde{\ul{d}}):=(\lambda_{ij}),\ n\geq i\geq j$
as follows:
$\lambda_{nj}:=\hbar^{-1}x_j+j-1,\ n\geq j\geq 1;\
\lambda_{ij}:=\hbar^{-1}x_j+j-1-d_{ij},\ n-1\geq i\geq j\geq1$.
Now we define $\fV\ni\xi_{\widetilde{\ul{d}}}=\xi_\Lambda:=
(-\hbar)^{-|\ul{d}|}\Psi[\widetilde{\ul{d}}]$.
According to Proposition~\ref{feigin'}, the matrix coefficients of the
operators $\fe_i,\ff_i$ in the basis $\{\xi_{\widetilde{\ul{d}}}\}$ are as
follows:

$$\fe_{i,\Lambda(\widetilde{\ul{d}}),\Lambda(\widetilde{\ul{d}}{}')}=
\prod_{j\ne k\leq i}(x_j-x_k+(d_{i,k}-d_{i,j})\hbar)^{-1}
\prod_{k\leq i-1}(x_j-x_k+(d_{i-1,k}-d_{i,j})\hbar)$$ if
$d'_{i,j}=d_{i,j}+1$ for certain $j\leq i$;

$$\ff_{i,\Lambda(\widetilde{\ul{d}}),\Lambda(\widetilde{\ul{d}}{}')}=
-\hbar^{-2}\prod_{j\ne k\leq i}(x_k-x_j+(d_{i,j}-d_{i,k})\hbar)^{-1}
\prod_{k\leq i+1}(x_k-x_j+(d_{i,j}-d_{i+1,k})\hbar)$$ if
$d'_{i,j}=d_{i,j}-1$ for certain $j\leq i$;

All the other matrix coefficients of $\fe_i,\ff_i$ vanish.

The above matrix coefficients, under appropriate specialization of
$x_1,\ldots,x_n$, coincide with the matrix coefficients
of $\fe_i,\ff_i$ in the {\em Gelfand-Tsetlin} basis of an
irreducible finite dimensional $\fgl_n$-module, cf. formulas~(2.7),~(2.6)
of~Theorem~2.3 of~\cite{m}. For this reason we suggest to call the basis
$\{\xi_{\widetilde{\ul{d}}}\}$ (over all collections $\widetilde{\ul{d}}$)
of $\fV$ the {\em Gelfand-Tsetlin basis}. Algebraically,
$\xi_{\widetilde{\ul{d}}}=\xi_\Lambda\in\fV$ can be defined
by the formulas (2.9)--(2.11) of~\cite{m} (where $\xi=\xi_0=1\in\fV$).
Up to proportionality, the Gelfand-Tsetlin basis can also be defined as an
eigenbasis of the Gelfand-Tsetlin subalgebra of $\fU$.

For a future reference, let us formulate once again the relation between
the fixed point base of $V$ and the Gelfand-Tsetlin base of $\fV$:

\begin{thm}
\label{feigin}
The isomorphism $\Psi:\ V\iso\fV$ of Theorem~\ref{brav} takes
$[\widetilde{\ul{d}}]$ to $(-\hbar)^{|\ul{d}|}\xi_{\widetilde{\ul{d}}}$
where $|\ul{d}|=d_1+\ldots+d_{n-1}$.
\end{thm}





\begin{rem}
{\em One can prove that the isomorphism
$\Psi:\ V\iso\fV$ of Theorem~\ref{brav} takes $[\widetilde{\ul{d}}]$
to $\xi_{\widetilde{\ul{d}}}$} up to proportionality {\em without explicitly
computing the matrix coefficients. In effect, the
Gelfand-Tsetlin basis is uniquely (up to proportionality) characterized
by the property that the matrix coefficients of $\fe_k,\ff_k$ with respect to
$\xi_\Lambda,\xi_{\Lambda'}$ vanish if $\lambda_{ij}\ne\lambda'_{ij}$ for some
$i>k$. Now it is immediate to see that the matrix coefficients of
$\fe_k,\ff_k$ with respect to $[\widetilde{\ul{d}}],[\widetilde{\ul{d}}{}']$
vanish if $d_{ij}\ne d'_{ij}$ for some $i>k$.
}
\end{rem}

\subsection{Determinant line bundles}
We consider the line bundle $\D_k$
on $\fQ_{\ul{d}}$ whose fiber at the point $(\CW_\bullet)$ equals
$\det R\Gamma(\bC,\CW_k)$.

\begin{lem}
\label{ias}
$\D_k$ is a $\widetilde{T}\times\BC^*$-equivariant line bundle, and
the character of $\widetilde{T}\times\BC^*$ acting in the fiber of
$\D_k$ at a point $\widetilde{\ul{d}}=(d_{ij})$ equals
$\sum_{j\leq k}(1-d_{kj})x_j+\frac{d_{kj}(d_{kj}-1)}{2}\hbar$.
\end{lem}

\begin{proof}
Straightforward.
\end{proof}

Let $Cas_k=\sum\limits_{i,j=1}^k E_{ij}E_{ji}$ be the quadratic Casimir
element of $U(\fgl_{k})$ naturally embedded into $U(\fgl_n)\subset\fU$.
The operator
$Cas_k$ is diagonal in the Gelfand-Tsetlin basis, and the eigenvalue of
$Cas_k$ on the basis vector $\xi_{\widetilde{\ul{d}}}=\xi_\Lambda$ is
$\sum_{j\leq k}\lambda_{kj}(\lambda_{kj}+k-2j+1)$. We define the following
element of $\fU$:
$$
\wt{Cas_k}:=Cas_k+(2-k)\sum\limits_{j=1}^k E_{jj}-\sum\limits_{j=1}^k
\hbar^{-1}x_j(\hbar^{-1}x_j-1)+
\frac{k(k-1)(k-2)}{3}.
$$
The eigenvalue of this element on the basis vector $\xi_{\widetilde{\ul{d}}}$
is $\sum_{j\leq k}2(1-d_{kj})x_j\hbar^{-1}+d_{kj}(d_{kj}-1)$.

\begin{cor}
\label{IAS}
a) The operator of multiplication by the first Chern class $c_1(\D_k)$
in $V$ is diagonal in the basis  $\{[\widetilde{\ul{d}}]\}$, and
the eigenvalue corresponding to $\widetilde{\ul{d}}=(d_{ij})$
equals $\sum_{j\leq k}(1-d_{kj})x_j+\frac{d_{kj}(d_{kj}-1)}{2}\hbar$.

b) The set $\{c_1(\D_k):\ k\geq2,\ d_k\ne0\ne d_{k-1}\}$ forms a basis
in the {\em nonequivariant} cohomology $H^2(\fQ_{\ul{d}})$.

c) The isomorphism $\Psi:\ V\iso\fV$ carries the operator of multiplication
by $c_1(\D_k)$ to the operator $\frac{\hbar}{2}\wt{Cas_k}$.
\end{cor}

\begin{proof}
a) follows from Lemma~\ref{ias}.

b) It follows e.g. from~\cite{fk} that
$\dim H^2(\fQ_{\ul{d}})=\sharp\{k\geq2,\ d_k\ne0\ne d_{k-1}\}$. Now it is
easy to see from Lemma~\ref{ias} that the classes
$\{[\D_k]:\ k\geq2,\ d_k\ne0\ne d_{k-1}\}$
in $\operatorname{Pic}(\fQ_{\ul{d}})$
are linearly independent, and hence the classes
$\{c_1(\D_k):\ k\geq2,\ d_k\ne0\ne d_{k-1}\}$ are linearly independent in
$H^2(\fQ_{\ul{d}})$.

c) Straightforward from a) and formula for eigenvalue of $\wt{Cas_k}$ on
$\xi_\Lambda$.

\end{proof}

\subsection{Gelfand-Tsetlin subalgebra and cohomology rings}
It is known that a completion $\widehat\fV$ of the universal Verma module
$\fV$ contains
a unique {\em Whittaker vector} $\fv=\sum_{\ul{d}}\fv_{\ul{d}}$ such that
$\fv_0=1$ (the lowest weight vector), and $\ff_i\fv=\hbar^{-1}\fv$ for any
$1\leq i\leq n-1$. Let us denote by
$1_{\ul{d}}\in H^0_{\widetilde{T}\times\BC^*}(\fQ_{\ul{d}})\subset V_{\ul{d}}$
the unit element of the cohomology ring. Then $\Psi(1_{\ul{d}})=\fv_{\ul{d}}$.
The proof is entirely similar to the proof of Proposition 2.31 of~\cite{bf},
and goes back to~\cite{b}.

Recall that the Gelfand-Tsetlin subalgebra $\fG\subset\operatorname{End}(\fV)$
is generated by the Harish-Chandra centers of the universal enveloping algebras
$\fgl_1,\fgl_2,\ldots,\fgl_n$ (embedded into $\fgl_n$ as the
upper left blocks) over the field $\BC(\ft\oplus\BC)$.
We denote by $\fI_{\ul{d}}\subset\fG$ the annihilator ideal of the vector
$\fv_{\ul{d}}\in\fV$, and we denote by $\fG_{\ul{d}}$ the quotient algebra
of $\fG$ by $\fI_{\ul{d}}$. The action of $\fG$ on $\fv_{\ul{d}}$ gives rise
to an embedding $\fG_{\ul{d}}\hookrightarrow\fV_{\ul{d}}$.

\begin{prop}
\label{gc}
a) $\fG_{\ul{d}}\iso\fV_{\ul{d}}$.

b) The composite morphism
$\Psi:\ H^\bullet_{\widetilde{T}\times\BC^*}(\fQ_{\ul{d}})
\otimes_{\BC[\ft\oplus\BC]}\BC(\ft\oplus\BC)=
V_{\ul{d}}\iso\fV_{\ul{d}}\iso\fG_{\ul{d}}$
is an {\em algebra} isomorphism.

c) The algebra $H^\bullet_{\widetilde{T}\times\BC^*}(\fQ_{\ul{d}})
\otimes_{\BC[\ft\oplus\BC]}\BC(\ft\oplus\BC)$ is generated by
$\{c_1(\D_k):\ k\geq2,\ d_k\ne0\ne d_{k-1}\}$.
\end{prop}

\begin{proof}
c) The algebra $H^\bullet_{\widetilde{T}\times\BC^*}(\fQ_{\ul{d}})
\otimes_{\BC[\ft\oplus\BC]}\BC(\ft\oplus\BC)$ consists of operators on the
space $V_{\ul{d}}$ which are diagonal in the basis of fixed points
$[\widetilde{\ul{d}}]$. On the other hand, the operators $Cas_k\in\fG$,
$k\geq2$, are diagonal in the Gelfand-Tsetlin basis $\xi_{\widetilde{\ul{d}}}$
and have different joint eigenvalues on different $\xi_{\widetilde{\ul{d}}}$.
Hence the images of $Cas_k$ in $\operatorname{End}(\fV_{\ul{d}})$ generate
the algebra of operators which are diagonal in the Gelfand-Tsetlin basis,
and in particular, the images of $Cas_k$, $k\geq2$, in $\fG_{\ul{d}}$
generate $\fG_{\ul{d}}$. By Theorem~\ref{feigin}, the isomorphism
$\Psi:\ V_{\ul{d}}\to\fV_{\ul{d}}$ carries $[\widetilde{\ul{d}}]$ to
$(-\hbar)^{|\ul{d}|}\xi_{\widetilde{\ul{d}}}$. By Corollary~\ref{IAS},
$c_1(\D_k)$ is
$\Psi^{-1}(\frac{\hbar}{2}Cas_k)$ up to an additive constant. Hence the elements
$c_1(\D_k)=
\Psi^{-1}(\frac{\hbar}{2}Cas_k)+const\in H^\bullet_{\widetilde{T}\times\BC^*}
(\fQ_{\ul{d}})
\otimes_{\BC[\ft\oplus\BC]}\BC(\ft\oplus\BC)$ generate the algebra
$H^\bullet_{\widetilde{T}\times\BC^*}(\fQ_{\ul{d}})
\otimes_{\BC[\ft\oplus\BC]}\BC(\ft\oplus\BC)$.

a-b) Since $\Psi(1_{\ul{d}})=\fv_{\ul{d}}$, the (surjective) homomorphism
$\Psi^{-1}:\BC[Cas_2,\dots,Cas_{n-1}]\to
H^\bullet_{\widetilde{T}\times\BC^*}(\fQ_{\ul{d}})
\otimes_{\BC[\ft\oplus\BC]}\BC(\ft\oplus\BC)$ factors through $\fG_{\ul{d}}$.
Hence (a) and (b).
\end{proof}

\section{Integral forms}

\subsection{Renormalized Universal Enveloping Algebra}
We denote by $\ul{\fU}\subset\fU$ the
$\BC[\ft\oplus\BC]$-subalgebra generated by the set
$\{\ul{E}{}_{ij}:=\hbar E_{ij},\ 1\leq i<j\leq n;\ E_{ij},\
1\leq j<i\leq n;\ E_{ii}':=E_{ii}-\hbar^{-1}x_i,\ i=1,\ldots,n\}$.
We denote by $\ul{\fU}{}_{\leq0}$ the subalgebra
of $\ul{\fU}$ generated by $\{E_{ii}',\ 1\leq i\leq n;\ \ul{E}{}_{ij},\
1\leq i<j\leq n\}$. It acts on the ring $\BC[\ft\oplus\BC]$ as follows:
$\ul{E}{}_{ij}$ acts trivially for any $i<j$, and $E_{ii}'$ acts
by multiplication by $i-1$. We define the integral
form of the universal Verma module $\ul{\fV}\subset\fV$ over $\ul\fU$ as
$\ul{\fV}:=\ul{\fU}\otimes_{\ul{\fU}{}_{\leq0}}\BC[\ft\oplus\BC]$.
We define the integral form of the universal dual Verma module
$\ul{\fV}^*\subset\fV$ as $\ul{\fV}^*:=\{u\in\fV:\ (u,u')\in\BC[\ft\oplus\BC]$
for any $u'\in\ul{\fV}\}$ (where $(u,u')$
stands for the Shapovalov form).
Clearly, $\ul{\fV}^*$ is a $\ul{\fU}$-module.

Note that the Whittaker vector $\fv\in\widehat\fV$ lies inside
the completion of $\ul{\fV}^*$, and
is uniquely characterized by the properties a) $\ul{\ff}{}_i\fv=\fv$
where $\ul{\ff}{}_i:=\ul{E}{}_{i,i+1}$; b) the highest weight component of
$\fv$ equals $1\in\BC[\ft\oplus\BC]$.

Finally, we denote by $\ul{\fG}\subset\fG$ the integral form of the
Gelfand-Tsetlin
subalgebra, generated by the centers of the algebras $\ul{\fU}{}_1,
\ul{\fU}{}_2,\ldots,\ul{\fU}{}_n=\ul{\fU}$ constructed from the Lie
algebras $\fgl_1,\fgl_2,\ldots,\fgl_n$ (embedded into $\fgl_n$ as the
upper left blocks) the same way as $\ul\fU$ is constructed from $\fgl_n$.
Recall that the Harish-Chandra isomorphism identifies the center of
$\fU(\fgl_k)$
with the ring of symmetric polynomials in $k$ variables. Namely, to any
symmetric
polynomial $P$ one assigns a central element $HC(P)$, whose PBW degree
equals $\deg P$, acting on the Verma module
with the highest weight $\lambda=(\lambda_1,\ldots,\lambda_k)$ as the
scalar operator with the eigenvalue
$P(\lambda_1,\ldots,\lambda_i-i+1,\ldots,\lambda_k-k+1)$.
Clearly, the central element $\ul{HC}(P):=\hbar^{\deg P}HC(P)$ lies in
$\ul{\fU}(\fgl_k)$. Moreover, the difference $\ul{HC}(P)-P(x_1,\ldots,x_k)$
is divisible by $\hbar$ in $\ul{\fU}(\fgl_k)$, hence
$\hbar^{-1}(\ul{HC}(P)-P(x_1,\ldots,x_k))$ also lies in the center of
$\ul{\fU}(\fgl_k)$.

We denote by $\ul{\fI}{}_{\ul{d}}\subset\ul\fG$
the annihilator ideal of the vector $\fv_{\ul{d}}\in\ul{\fV}^*$, and
we denote by $\ul{\fG}{}_{\ul{d}}$ the quotient algebra of $\ul\fG$ by
$\ul\fI$. The action of $\ul\fG$ on $\fv_{\ul{d}}$ gives rise to an embedding
$\ul{\fG}{}_{\ul{d}}\hookrightarrow\ul{\fV}{}^*_{\ul{d}}$.

\begin{lem}
\label{Bekasovo}
$\ul{\fG}{}_{\ul{d}}\iso\ul{\fV}{}^*_{\ul{d}}$.
\end{lem}

\begin{proof}
By graded Nakayama lemma, it suffices to prove the surjectivity of
$\ul{\fG}{}_{\ul{d}}/(x_1,\ldots,x_n,\hbar=0)\to
\ul{\fV}{}^*_{\ul{d}}/(x_1,\ldots,x_n,\hbar=0)$.
We denote by $\fg_{>0}\subset\fgl_n$ the Lie subalgebra spanned by the
set $\{E_{ij},\ 1\leq j<i\leq n\}$. We denote by
$\fg_{\geq0}\subset\fgl_n$ (resp. $\fg_{\leq0}\subset\fgl_n$,
$\fg_{>0}\subset\fgl_n$, $\fg_{<0}\subset\fgl_n$)
the Lie subalgebra spanned by the
set $\{E_{ij},\ 1\leq j\leq i\leq n\}$
(resp. $\{E_{ij},\ 1\leq i\leq j\leq n\}$, $\{E_{ij},\ 1\leq j<i\leq n\}$,
$\{E_{ij},\ 1\leq i< j\leq n\}$). The Killing form identifies
the vector space $\fg_{>0}$ with the dual of $\fg_{<0}$, and
gives rise to an isomorphism $\on{Sym}(\fg_{<0})\simeq\BC[\fg_{>0}]$.
The universal enveloping algebra of $\fg_{\geq0}$ over $\BC[\ft]$ lies
inside $\ul\fU$ and is denoted by $\ul{\fU}{}_{\geq0}$. Evidently,
$\ul{\fU}{}_{\geq0}\simeq U(\fg_{\geq0})\otimes\BC[\ft]$. We have
$\ul{\fU}/(x_1,\ldots,x_n,\hbar=0)\simeq\BC[\fg_{>0}]\rtimes U(\fg_{\geq0})$.
Here the semidirect product is taken with respect to the adjoint action
of $\fg_{\geq0}$ on $\fg_{>0}$ (and the induced action on the algebra of
functions).

Let $\sV$ denote the space of distributions on $\fg_{>0}$ supported at
the origin, that is cohomology with support of the structure sheaf
$H^{\frac{n(n-1)}{2}}_{\{0\}}(\fg_{>0},\CO)$. The algebra
$\BC[\fg_{>0}]\rtimes U(\fg_{\geq0})$ acts on $\sV$
naturally. As a $\BC[\fg_{>0}]$-module, $\sV$ is cofree,
and its completion is naturally isomorphic to
$\BC[\fg_{>0}]^*$. Clearly, $\ul{\fV}{}^*_{\ul{d}}/(x_1,\ldots,x_n,\hbar=0)$
as a module over
$\ul{\fU}/(x_1,\ldots,x_n,\hbar=0)\simeq\BC[\fg_{>0}]\rtimes U(\fg_{\geq0})$
is isomorphic to $\sV$.
The value of the Whittaker vector $\fv|_{\hbar=0}$ in the completion of
$\sV$ is the functional $\chi:\BC[\fg_{>0}]\to\BC$ which sends
$P\in\BC[\fg_{>0}]$ to $P(\ff)\in\BC$, where
$\ff=\sum\limits_{i=1}^{n-1}\ff_i$ is the principal nilpotent element.
The adjoint $G_{\geq0}$-orbit of $\ff$ is dense in $\fg_{>0}$, hence the
submodule generated by the Whittaker vector is dense in the completion of
$\sV$. This means that each weight space of $\sV$ is generated by the
component of the Whittaker vector.

Consider the Whittaker module $W$ over $\BC[\fg_{>0}]\rtimes U(\fg_{\geq0})$,
that is, induced from the character $\chi:\BC[\fg_{>0}]\to\BC$. The module
$W$ is free with respect to $U(\fg_{\geq0})$, and hence has natural
filtration coming from the PBW filtration on $U(\fg_{\geq0})$. The
associated graded $\gr W$ is naturally a $\BC[\fg_{\geq0}]\otimes
S(\fg_{>0})=\BC[\fg]$-module. It is easy to see that
$\gr W = \BC[\ff+\fg_{\geq0}]$. The restriction of the Gelfand-Tsetlin
subalgebra in $\BC[\fg]$ to the affine subspace $\ff+\fg_{\geq0}$ is
known to be an isomorphism onto $\BC[\ff+\fg_{\geq0}]$
(see~\cite{KW},~\cite{Tar}). Thus the module $W$ is generated by the
Whittaker vector as a $\ul{\fG}{}_{\ul{d}}/(x_1,\ldots,x_n,\hbar=0)$-module.
Hence the $\ul{\fG}{}_{\ul{d}}/(x_1,\ldots,x_n,\hbar=0)$-submodule generated
by the Whittaker vector $\fv$ is dense in the completion of $\sV$. This means
that each weight space of $\sV$ is generated by the component of the Whittaker
vector with respect to the action of the Gelfand-Tsetlin subalgebra.
\end{proof}

\subsection{Kuznetsov correspondences}
\label{kuzn}
We consider a correspondence
$\ul{\fE}{}_{\ul{d},i}\subset\fQ_{\ul{d}}\times\fQ_{\ul{d}+i}$ defined
as $\bfr^{-1}\{0\}$. In other words, $\ul{\fE}{}_{\ul{d},i}$ is a closed
subvariety of $\fE_{\ul{d},i}$ where we impose a condition that the quotient
flag is supported at $\{0\}\in\bC$. It is a smooth quasiprojective variety
of dimension $2d_1+\ldots+2d_{n-1}$. We denote by
$\ul{\bp}:\ \ul{\fE}{}_{\ul{d},i}\to\fQ_{\ul{d}},\
\ul{\bq}:\ \ul{\fE}{}_{\ul{d},i}\to\fQ_{\ul{d}+i}$ the natural projections.
Note that both $\ul\bp$ and $\ul\bq$ are proper.

More generally, for $1\leq i<j\leq n$ we denote by $\ul{d}\pm\alpha_{ij}$
the sequence
$(d_1,\ldots,d_{i-1},d_i\pm1,\ldots,d_{j-1}\pm1,d_j,\ldots,d_{n-1})$.
We have a correspondence $^\circ\ul{\fE}{}_{\ul{d},\alpha_{ij}}\subset
\fQ_{\ul{d}}\times\fQ_{\ul{d}+\alpha_{ij}}$ formed by the pairs
$(\CW_\bullet,\CW'_\bullet)$ such that a) $\CW'_k\subset\CW_k$ for any
$1\leq k\leq n-1$; b) The quotient $\CW_\bullet/\CW'_\bullet$ is supported
at $\{0\}\in\bC$; c) For $i\leq k<j$ the natural map $\CW_k/\CW'_k\to
\CW_{k+1}/\CW'_{k+1}$ is an isomorphism (of one-dimensional vector spaces).
We define a correspondence $\ul{\fE}{}_{\ul{d},\alpha_{ij}}\subset
\fQ_{\ul{d}}\times\fQ_{\ul{d}+\alpha_{ij}}$ as the closure of
$^\circ\ul{\fE}{}_{\ul{d},\alpha_{ij}}$. According to Lemma~5.2.1 of~\cite{fk},
$\ul{\fE}{}_{\ul{d},\alpha_{ij}}$ is irreducible of dimension
$2d_1+\ldots+2d_{n-1}+j-i-1$. We denote by
$\ul{\bp}{}_{ij}:\ \ul{\fE}{}_{\ul{d},\alpha_{ij}}\to\fQ_{\ul{d}},\
\ul{\bq}{}_{ij}:\ \ul{\fE}{}_{\ul{d},\alpha_{ij}}\to\fQ_{\ul{d}+\alpha_{ij}}$
the natural projections.
Note that both $\ul{\bp}{}_{ij}$ and $\ul{\bq}{}_{ij}$ are proper.

Also, we consider the correspondences
$\fE_{\ul{d},\alpha_{ij}}\subset
\fQ_{\ul{d}}\times\fQ_{\ul{d}+\alpha_{ij}}$ defined exactly as
$\ul{\fE}{}_{\ul{d},\alpha_{ij}}\subset
\fQ_{\ul{d}}\times\fQ_{\ul{d}+\alpha_{ij}}$ in~\ref{kuzn} but with
condition b) removed (i.e. we allow the quotient flag to be supported
at an arbitrary point of $\bC-\infty$).
In particular, $\fE_{\ul{d},\alpha_{i,i+1}}=
\fE_{\ul{d},i}$. We denote by
$\bp_{ij}:\ \fE_{\ul{d},\alpha_{ij}}\to\fQ_{\ul{d}},\
\bq_{ij}:\ \fE_{\ul{d},\alpha_{ij}}\to\fQ_{\ul{d}+\alpha_{ij}}$
the natural projections. Note that $\bq_{ij}$ is proper, while
$\bp_{ij}$ is not.

\subsection{The action of the renormalized Universal Enveloping Algebra}
\label{kuz}
Recall that $'V=\oplus_{\ul{d}}\ 'V_{\ul{d}}:=
\oplus_{\ul{d}}H^\bullet_{\widetilde{T}\times\BC^*}(\fQ_{\ul{d}})$.
The grading and the correspondences $\ul{\fE}{}_{\ul{d},\alpha_{ij}}$
give rise to the following operators on $'V$:

$\ul{E}{}_{ii}=x_i+(d_{i-1}-d_i+i-1)\hbar:\ 'V_{\ul{d}}\to\
'V_{\ul{d}}$;

$\ul{\fh}{}_i=(x_{i+1}-x_i)+(2d_i-d_{i-1}-d_{i+1}+1)\hbar:\ 'V_{\ul{d}}\to\
'V_{\ul{d}}$;

$\ul{\ff}{}_i=\ul{E}{}_{i,i+1}=\ul{\bp}{}_*\ul{\bq}{}^*:\ 'V_{\ul{d}}\to\
'V_{\ul{d}-i}$;

$\ul{\fe}{}_i=\ul{E}{}_{i+1,i}=-\ul{\bq}{}_*\ul{\bp}{}^*:\
'V_{\ul{d}}\to\ 'V_{\ul{d}+i}$;

$\ul{E}{}_{ij}=\ul{\bp}{}_{ij*}\ul{\bq}{}_{ij}^*:\ 'V_{\ul{d}}\to\
'V_{\ul{d}-\alpha_{ij}}\ (1\leq i<j\leq n)$;

$\ul{E}{}_{ji}=(-1)^{j-i}\ul{\bq}{}_{ij*}\ul{\bp}{}_{ij}^*:\
'V_{\ul{d}}\to\ 'V_{\ul{d}+\alpha_{ij}}\ (1\leq i<j\leq n)$;

$E_{ji}=(-1)^{j-i}\bq_{ij*}\bp_{ij}^*:\
'V_{\ul{d}}\to\ 'V_{\ul{d}+\alpha_{ij}}\ (1\leq i<j\leq n)$.

\begin{thm}
\label{braverman}
a) The operators $\{\ul{E}{}_{ij},\ 1\leq i\leq j\leq n;\ E_{ij},\
1\leq j<i\leq n\}$ on $'V$ defined
in~\ref{kuz} satisfy the relations in $\ul{\fU}$, i.e. they give rise to
the action of $\ul\fU$ on $'V$.

b) There is a unique isomorphism $\Phi$ of $\ul\fU$-modules $'V$ and
$\ul{\fV}^*$ carrying $1\in H^0_{\widetilde{T}\times\BC^*}(\fQ_0)\subset\ 'V$
to the lowest weight vector $1\in\BC[\ft\oplus\BC]\subset\ul{\fV}^*$.
\end{thm}

\begin{proof}
a) We define the operators

\begin{equation}
\label{kuzne}
E_{ij}=\bp_{ij*}\bq_{ij}^*:\ V_{\ul{d}}\to
V_{\ul{d}-\alpha_{ij}}\ (1\leq i<j\leq n).
\end{equation}

It is clear that $\fE_{\ul{d},\alpha_{ij}}\simeq
\ul{\fE}{}_{\ul{d},\alpha_{ij}}\times(\bC-\infty)$. It follows that
for any $1\leq i,j\leq n$ we have $\ul{E}{}_{ij}=\hbar E_{ij}$.
Furthermore, the operators $E_{i,i\pm1}$ are exactly those defined
in~\ref{operators}, and they satisfy the relations of $\fU$
(and generate it) by
Theorem~\ref{brav}. Finally, according to Proposition~5.6 of~\cite{fk},
the elements $E_{ij}\in\fU,\ 1\leq i\ne j\leq n$ act in $V$ by the same
named operators of~(\ref{kuzne}) and~\ref{kuz}. This proves a).

b) We have $'V\subset V\iso\fV\supset\ul{\fV}^*$, so we have to check
that $\Psi(\ 'V)=\ul{\fV}^*$, and then $\Phi=\Psi|_{'V}$.
Recall that
$\Psi(\fv_{\ul{d}})=1\in H^0_{\widetilde{T}\times\BC^*}(\fQ_{\ul{d}})$.
By the virtue of Lemma~\ref{Bekasovo} it suffices to prove that
$H^\bullet_{\widetilde{T}\times\BC^*}(\fQ_{\ul{d}})$ is generated
by the action of the integral form $\ul\fG$ of the Gelfand-Tsetlin subalgebra
on the vector $1\in H^0_{\widetilde{T}\times\BC^*}(\fQ_{\ul{d}})$.

For any $1\leq i\leq n-1$ we will denote by $\ul{\CW}{}_i$ the tautological
$i$-dimensional vector bundle on $\fQ_{\ul{d}}\times\bC$. By the K\"unneth
formula we have
$H^\bullet_{\widetilde{T}\times\BC^*}(\fQ_{\ul{d}}\times\bC)=
H^\bullet_{\widetilde{T}\times\BC^*}(\fQ_{\ul{d}})\otimes1\oplus
H^\bullet_{\widetilde{T}\times\BC^*}(\fQ_{\ul{d}})\otimes\tau$ where
$\tau\in H^2_{\BC^*}(\bC)$ is the first Chern class of $\CO(1)$.
Under this decomposition, for the Chern class $c_j(\ul{\CW}{}_i)$ we have
$c_j(\ul{\CW}{}_i)=:c_j^{(j)}(\ul{\CW}{}_i)\otimes1+
c_j^{(j-1)}(\ul{\CW}{}_i)\otimes\tau$
where $c_j^{(j)}(\ul{\CW}{}_i)\in
H^{2j}_{\widetilde{T}\times\BC^*}(\fQ_{\ul{d}})$, and
$c_j^{(j-1)}(\ul{\CW}{}_i)\in
H^{2j-2}_{\widetilde{T}\times\BC^*}(\fQ_{\ul{d}})$.

The following Lemma goes back to~\cite{be}\footnote{We
have learnt of it from A.~Marian.}:

\begin{lem}
\label{beau}
The equivariant cohomology ring
$H^\bullet_{\widetilde{T}\times\BC^*}(\fQ_{\ul{d}})$
is generated by the classes $c_j^{(j)}(\ul{\CW}{}_i),\
c_j^{(j-1)}(\ul{\CW}{}_i),\ 1\leq j\leq i\leq n-1$ (over the algebra
$\BC[\ft\oplus\BC]$).
\end{lem}

\begin{proof}
By the graded Nakayama lemma, it suffices to prove that the nonequivariant
cohomology ring $H^\bullet(\fQ_{\ul{d}})$ is generated by the K\"unneth
components of the (nonequivariant) Chern classes $c_j^{(j)}(\ul{\CW}{}_i),\
c_j^{(j-1)}(\ul{\CW}{}_i),\ 1\leq j\leq i\leq n-1$. The locally closed
embedding $\fQ_{\ul{d}}\hookrightarrow\CQ_{\ul{d}}$ induces a surjection on
the cohomology rings (see e.g. the computation of cohomology of $\CQ_{\ul{d}}$
in~\cite{fk}), so it suffices to prove that the cohomology ring of
the compact smooth variety $H^\bullet(\CQ_{\ul{d}})$ is generated by the
K\"unneth components of the Chern classes of the tautological bundles.
But this follows from~Theorem~2.1 of~\cite{be},
since the fundamental class of the diagonal
in $\CQ_{\ul{d}}\times\CQ_{\ul{d}}$ can be expressed via the Chern classes
of the tautological vector bundles (cf.~\cite{neg},~section~5).
\end{proof}

Returning to the proof of the theorem, if suffices to check that the operators
of multiplication by $c_j^{(j)}(\ul{\CW}{}_i),\
c_j^{(j-1)}(\ul{\CW}{}_i),\ 1\leq j\leq i\leq n-1$, in the equivariant
cohomology ring $H^\bullet_{\widetilde{T}\times\BC^*}(\fQ_{\ul{d}})=\
'V_{\ul{d}}$ lie in the integral form $\ul\fG$ of the Gelfand-Tsetlin
subalgebra. To this end we compute these operators explicitly in the
fixed point basis $\{[\widetilde{\ul{d}}]\}$ (alias Gelfand-Tsetlin basis
$\{\xi_{\widetilde{\ul{d}}}\})$ of $V_{\ul{d}}=\fV_{\ul{d}}$.

The set of eigenvalues of $\ft\oplus\BC$ in the fiber of
$\ul\CW{}_i$ at a point $(\widetilde{\ul{d}},\infty)$
(resp. $(\widetilde{\ul{d}},0)$) equals $\{-x_1,\ldots,-x_i\}$
(resp. $\{-x_1+d_{i,1}\hbar,\ldots,-x_i+d_{i,i}\hbar\}$).
For $1\leq j\leq i$, let $e_{ji}^\infty$ (resp. $e_{ji}^0(\widetilde{\ul{d}})$)
stand for the sum of products of the
$j$-tuples of distinct elements of the set $\{-x_1,\ldots,-x_i\}$
(resp. of the set $\{-x_1+d_{i,1}\hbar,\ldots,-x_i+d_{i,i}\hbar\}$).
Then the operator of multiplication by the Chern class $c_j(\ul\CW{}_i)$
is diagonal in the basis $\{[\widetilde{\ul{d}},\infty],
[\widetilde{\ul{d}},0]\}$
with eigenvalues $\{e_{ji}^\infty,e_{ji}^0(\widetilde{\ul{d}})\}$.
It follows that the operator
of multiplication by $c_j^{(j)}(\ul{\CW}{}_i)$ (resp. by
$c_j^{(j-1)}(\ul{\CW}{}_i),\ 1\leq j\leq i\leq n-1$) is diagonal in the
basis $\{[\widetilde{\ul{d}}]\}$ with eigenvalues
$\{\frac{1}{2}(e_{ji}^\infty+e_{ji}^0(\widetilde{\ul{d}}))\}$ (resp.
$\{\frac{\hbar^{-1}}{2}(e_{ji}^\infty-e_{ji}^0(\widetilde{\ul{d}}))\}$).
Note that $e_{ji}^\infty\in\BC[\ft\oplus\BC]$, and
$e_{ji}^0(\widetilde{\ul{d}}))$ is precisely the eigenvalue of
the central element $\ul{HC}(e_{ji})\in\ul\fU(\fgl_i)$
corresponding to the $j$-th elementary symmetric function
$e_j$ via the Harish-Chandra isomorphism, on the Verma module with
highest weight $\{\lambda_{i1}\hbar,\ldots,\lambda_{ii}\hbar\}$. Hence
the operator of multiplication by $c_j^{(j)}(\ul{\CW}{}_i)$
(with eigenvalues $\{\frac{1}{2}(e_{ji}^\infty+
e_{ji}^0(\widetilde{\ul{d}}))\}$) lies in the integral form $\ul\fG$
of the Gelfand-Tsetlin
subalgebra. Moreover, $e_{ji}^\infty-\ul{HC}(e_{ji})$ is divisible by
$\hbar$ in $\ul{\fU}(\fgl_i)$. Hence the operator of multiplication by
$c_j^{(j-1)}(\ul{\CW}{}_i)$ lies in $\ul\fG$ as well.
\end{proof}

\begin{cor}
\label{mor}
The composition of isomorphisms
$\ul{\fG}{}_{\ul{d}}\iso\ul{\fV}{}^*_{\ul{d}}\stackrel{\Phi^{-1}}
{\longleftarrow}H^\bullet_{\widetilde{T}\times\BC^*}(\fQ_{\ul{d}})$
is an isomorphism of algebras.
\end{cor}   \qed

\section{Speculation on equivariant quantum cohomology of $\fQ_{\ul{d}}$}

\subsection{Calabi-Yau property of Laumon spaces}
According to Theorem~3 of~\cite{gl}, the variety $\fQ_{\ul{d}}$ is Calabi-Yau.
For the reader's convenience we recall a proof. First note that if
$\ul{d}=(d_1,\ldots,d_{n-1})$, and $d_k=0$, then
$\fQ_{\ul{d}}=\fQ_{\ul{d}'}\times\fQ_{\ul{d}''}$ where
$\ul{d}'=(d_1,\ldots,d_{k-1})$, and $\fQ_{\ul{d}'}$ is the corresponding
Laumon moduli space for $GL_k$, while
$\ul{d}''=(d_{k+1},\ldots,d_{n-1})$, and $\fQ_{\ul{d}''}$ is the corresponding
Laumon moduli space for $GL_{n-k}$. Hence we may assume that all the
integers $d_1,\ldots,d_{n-1}$ are strictly positive.

For $1\leq i\leq n-1$, we consider the locally closed subvariety of
$\fQ_{\ul{d}}$ formed by all the quasiflags which have a defect of degree
exactly $i$. We denote by $\fD_i\subset\fQ_{\ul{d}}$ the closure of this
subvariety. It is a divisor. We denote by $[\CO(\fD_i)]$ the class of the
corresponding line bundle in $\operatorname{Pic}(\fQ_{\ul{d}})$.

\begin{lem}
\label{plusminus}
Assume all the integers $d_1,\ldots,d_{n-1}$ are strictly positive. Then
$[\CO(\fD_1)]=[\D_2],\ldots,[\CO(\fD_i)]=[\D_{i+1}]-[\D_i],\ldots,
[\CO(\fD_{n-1})]=-[\D_{n-1}]$.
\end{lem}

\begin{proof}
Recall the morphism $\pi_{\ul{d}}:\ \fQ_{\ul{d}}\to Z_{\ul{d}}$ to
Drinfeld's Zastava space (a small resolution of singularities).
For $1\leq k\leq n-2$, choose a point $s\in Z_{\ul{d}}$ with defect
of degree exactly $k+(k+1)$. We denote by $P_k$ the preimage
$\pi_{\ul{d}}^{-1}(s)\subset\fQ_{\ul{d}}$. By a $GL_3$-calculation,
$P_k$ is a projective line. The fundamental classes $[P_1],\ldots,[P_{n-2}]$
form a basis of $H_2(\fQ_{\ul{d}},\BC)$. It is easy to see that the restriction
$\D_i|_{P_k}$ is trivial if $i\ne k+1$, while $\D_{k+1}|_{P_k}\simeq\CO(1)$
(this is again a $GL_3$-calculation). Furthermore, it is easy to see that
the restriction $\CO(\fD_i)|_{P_k}$ is trivial for $i\ne k,k+1$, while
$\CO(\fD_k)|_{P_k}\simeq\CO(1)$, and $\CO(\fD_{k+1})|_{P_k}\simeq\CO(-1)$
(once again a $GL_3$-calculation). The lemma is proved.
\end{proof}

Let $^\circ\fQ_{\ul{d}}\subset\fQ_{\ul{d}}$ denote the open subspace formed
by all the quasiflags without defect. Recall the symplectic form $\Omega$
on $^\circ\fQ_{\ul{d}}$ constructed in~\cite{fkmm}. Note that the complement
$\fQ_{\ul{d}}\setminus\ ^\circ\fQ_{\ul{d}}$ equals the union of divisors
$\bigcup_{1\leq i\leq n-1}\fD_i$. Thus the top power $\omega:=\Omega^{top}$
is a meromorphic volume form on $\fQ_{\ul{d}}$ with poles at
$\bigcup_{1\leq i\leq n-1}\fD_i$. The formula for $\Omega$ given in
Remark~3 of {\em loc. cit.} shows that the order of the pole of $\omega$
at $\fD_i$ is 2 for any $1\leq i\leq n-1$. Hence the canonical class of
$\fQ_{\ul{d}}$ in $\operatorname{Pic}(\fQ_{\ul{d}})$ equals
$2\sum_{1\leq i\leq n-1}[\CO(\fD_i)]$. By Lemma~\ref{plusminus}, the class
$2\sum_{1\leq i\leq n-1}[\CO(\fD_i)]=0$. We have proved

\begin{cor}
\label{givlee}
(Givental, Lee~\cite{gl}) The canonical class of $\fQ_{\ul{d}}$ is trivial.
\end{cor}

\subsection{Shift of argument subalgebras}
To each regular element $\mu$ of the Cartan subalgebra $\hh$ of a
semisimple Lie algebra $\gg$ one can assign a space $Q_\mu$ of commuting
quadratic elements of the universal enveloping algebra $U(\gg)$. Namely,
$$
Q_\mu=\{\sum\limits_{\alpha\in\Delta_+}\frac{\langle h,\alpha\rangle}
{\langle \mu,\alpha\rangle}e_\alpha e_{-\alpha},\ |\ h\in\hh\}, $$
where $\Delta_+$ is the set of positive roots, and $e_\alpha,\ e_{-\alpha}$
are nonzero elements of the root spaces such that $(e_\alpha,e_{-\alpha})=1$.
Note that the space $Q_\mu$ does not change under dilations of $\mu$,
hence we have a family of spaces of commuting quadratic operators,
parametrized by the regular part of $\mathbb{P}(\hh)$.

These quadratic elements appear as the quasiclassical limit of the
\emph{Casimir} flat connection on the
trivial bundle on the regular part of the Cartan subalgebra with the
fiber $\fV_{\ul{d}}$, (cf.~\cite{dc,fmtv,TL,MTL}).
This connection is given by the
formula
$$
\nabla = d + \kappa\sum\limits_{\alpha\in\Delta_+}e_\alpha
e_{-\alpha}\frac{d\alpha}{\alpha},
$$
where $\kappa$ is a parameter. Since every element of $U(\gg)$ of the
form $e_\alpha e_{-\alpha}$ commutes with the Cartan subalgebra $\hh$,
this connection remains flat after adding any closed $U(\hh)$-valued $1$-form.

The centralizer in $U(\gg)$ of this space of commuting quadratic operators
is the so-called \emph{shift of argument subalgebra}
$\A_\mu\subset U(\gg)$, which
is a free commutative subalgebra with $\frac{1}{2}(\dim\gg+\rk\gg)$
generators. For $\gg=\fsl_n$ the family of commutative subalgebras
$\A_\mu\subset U(\gg)$ is an $(n-2)$-parametric deformation of the
Gelfand-Tsetlin subalgebra (see \cite{Ryb1,Ryb2}).

\subsection{Conjecture on equivariant quantum cohomology}
Consider the shift of argument subalgebra for $\gg=\fgl_n$, $\mu=\sum\limits_{i=1}^{n-1}q_{i+1}q_{i+2}\dots q_{n}\omega_i$ with $\omega_i$ being the fundamental weights of $\fgl_n$. Since the shift of argument algebra does not change under dilations of $\mu$, we can assume that $q_n=1$. Taking $h_k=\sum\limits_{i=1}^{k-1}q_iq_{i+1}\dots q_{n}\omega_i$, we find that the space $Q_\mu$ is generated by the elements
\begin{multline*}
\sum\limits_{\alpha\in\Delta_+}\frac{\langle h_k,\alpha\rangle}{\langle \mu,\alpha\rangle}e_\alpha e_{-\alpha} = \sum\limits_{i<j\leq k}E_{ij}E_{ji} + \sum\limits_{i<k<j} \frac{\sum\limits_{l=i+1}^{k}q_lq_{l+1}\dots q_{n}}{\sum\limits_{l=i+1}^{j}q_lq_{l+1}\dots q_{n}} E_{ij}E_{ji}=\\=\sum\limits_{i<j\leq k}E_{ij}E_{ji} + \sum\limits_{i<k<j} \frac{\sum\limits_{l=i+1}^{k}q_lq_{l+1}\dots q_{j-1}}{1+\sum\limits_{l=i+1}^{j-1}q_lq_{l+1}\dots q_{j-1}} E_{ij}E_{ji},
\end{multline*}
with $k=2,\dots,n-1$.

We consider the equivariant (small) quantum cohomology ring of $\fQ_{\ul{d}}$
which depends on $n-2$ quantum parameters $q_2,\dots,q_{n-1}$ corresponding
to the Chern classes of the determinant bundles. Note that
$\sum\limits_{i<j\leq k}E_{ij}E_{ji}$ is $Cas_k$ up to some Cartan term.
Hence the shift of argument subalgebra contains the following (commuting)
elements $$
QC_k:=\wt{Cas_k}+\sum\limits_{i<k<j}
\frac{\sum\limits_{l=i+1}^{k}q_lq_{l+1}\dots
q_{j-1}}{1+\sum\limits_{l=i+1}^{j-1}q_lq_{l+1}\dots q_{j-1}} E_{ij}E_{ji}.
$$

\begin{conj}
\label{Lenya}\footnote{It was recently proved by A.~Negut.}
The isomorphism $\Psi:V_{\ul{d}}\to\fV_{\ul{d}}$ carries the
operator $M_{\D_k}$ of quantum multiplication by $c_1(D_k)$ to the operator
$\frac{\hbar}{2}QC_k$.
\end{conj}

\begin{cor} The localized equivariant quantum cohomology ring of
$\fQ_{\ul{d}}$ is isomorphic to the quotient of the shift of argument
subalgebra $\A_\mu$ by the annihilator of $\fv_{\ul{d}}$.
\end{cor}

Let $\ul{\A}{}_\mu$ denote the integral form $\A_\mu\cap\ul\fU$.

\begin{conj}
\label{Len}
The equivariant quantum cohomology ring of
$\fQ_{\ul{d}}$ is isomorphic to the quotient of $\ul{\A}{}_\mu$ by the
annihilator of $\fv_{\ul{d}}$.
\end{conj}

\begin{rem} \emph{It is natural to expect~\ref{Len} since the analogue of
Lemma~\ref{Bekasovo} is valid for $\ul{\A}{}_\mu$ as well
(and the proof is the same).}
\end{rem}

The map $(q_2,\ldots,q_{n-1})\mapsto\mu=
\sum\limits_{i=1}^{n-1}q_{i+1}q_{i+2}\dots q_{n-1}\omega_i$
embeds the torus ${\mathsf T}=\BC^{*(n-2)}$ with coordinates
$q_2,\ldots,q_{n-1}$ into the Cartan subalgebra $\hh=\BC^{n-1}$ of $\fsl_n$
as an open subset of
an affine hyperplane. Restricting the Casimir connection to $\mathsf T$
and adding an appropriate Cartan term, we obtain the following flat
connection on $\mathsf T$ in the coordinates $q_i$:
$$
\nabla = d+\kappa\sum\limits_{k=2}^{n-1} QC_k\frac{dq_k}{q_k}.
$$
On the other hand, the trivial vector bundle with the fiber $V_{\ul{d}}$
over the space of quantum parameters $\mathsf T$
is equipped with the {\em quantum connection}
$d+\sum\limits_{k=2}^{n-1} M_{\D_k}
\frac{dq_k}{q_k}$.

\begin{cor} The isomorphism $\Psi$ carries the quantum connection to the
Casimir connection with $\kappa=\frac{\hbar}{2}$.
\end{cor}

\begin{rem} {\em According to Vinberg~\cite{Vin}, the family of subspaces
$Q_{\mu}$ form an open subset in the moduli space of $(n-1)$-dimensional
spaces of commuting linear combinations of $e_\alpha e_{-\alpha}$ in
$U(\gg)$. Thus it is natural to expect that the operators $M_{\D_k}$ span
the space $Q_{\mu}$ for some $\mu$ depending on $q_2,\ldots,q_{n-1}$ (but
unfortunately, we have no idea how to prove that the operators
$M_{\D_k}$ are quadratic expressions in the correspondences). Moreover,
since the unit in the cohomology ring remains unit in the quantum
cohomology ring, the quantum correction has to annihilate the vector
$\fv_{\ul{d}}$. Therefore the operator $\Psi(M_{\D_k})$ is $QC_k$ up to
a change of parametrization. Finally, the flatness of the quantum
connection $d+\sum\limits_{k=2}^{n-1} M_{\D_k}\frac{dq_k}{q_k}$ is a
very restrictive condition on the parametrization
$\mu(q_2,\ldots,q_{n-1})$ --- this leaves no other choice for
$\Psi(M_{\D_k})$ but to coincide with $\frac{\hbar}{2}QC_k$.}
\end{rem}

\section{Global Laumon spaces}
\label{GLS}

\subsection{Correspondences} 
Recall the setup of~\ref{classic}. We define two versions
of the correspondences $\CE_{\ul{d},i}\subset\CQ_{\ul{d}}\times\CQ_{\ul{d}+i}$,
namely, $\CE_{\ul{d},i}^0:=\bfr^{-1}\{0\},\
\CE_{\ul{d},i}^\infty:=\bfr^{-1}\{\infty\}$. Their projections to
$\CQ_{\ul{d}}$ (resp. $\CQ_{\ul{d}+i}$) will be denoted by $\bp^0,\bp^\infty$
(resp. $\bq^0,\bq^\infty$). The projection of $\CE_{\ul{d},i}$ to
$\CQ_{\ul{d}}$ (resp. $\CQ_{\ul{d}+i}$) will be denoted by $\bp^\bC$
(resp. $\bq^\bC$).

We denote by ${}'A$ (resp. ${}'B$)
the direct sum of equivariant (complexified) cohomology:
${}'A=\oplus_{\ul{d}}H^\bullet_{\widetilde{T}\times\BC^*}(\CQ_{\ul{d}})$
(resp. ${}'B=\oplus_{\ul{d}}H^\bullet_{G\times\BC^*}(\CQ_{\ul{d}})$).
We define $A=\ {}'A\otimes_{H^\bullet_{\widetilde{T}\times\BC^*}(pt)}
\on{Frac}(H^\bullet_{\widetilde{T}\times\BC^*}(pt))$, and
$B=\ {}'B\otimes_{H^\bullet_{G\times\BC^*}(pt)}
\on{Frac}(H^\bullet_{G\times\BC^*}(pt))$.

We have an evident grading $A=\oplus_{\ul{d}}A_{\ul{d}},\
A_{\ul{d}}=H^\bullet_{\widetilde{T}\times\BC^*}(\CQ_{\ul{d}})
\otimes_{H^\bullet_{\widetilde{T}\times\BC^*}(pt)}
\on{Frac}(H^\bullet_{\widetilde{T}\times\BC^*}(pt))$;
similarly, $B=\oplus_{\ul{d}}B_{\ul{d}},\
B_{\ul{d}}=H^\bullet_{G\times\BC^*}(\CQ_{\ul{d}})
\otimes_{H^\bullet_{G\times\BC^*}(pt)}
\on{Frac}(H^\bullet_{G\times\BC^*}(pt))$.

Note that $H^\bullet_{G\times\BC^*}(pt)=\BC[e_1,\ldots,e_n,\hbar]$ where
$e_i$ is the $i$-th elementary symmetric function of $x_1,\ldots,x_n$.

\subsection{Fixed points}
The set of $\widetilde{T}\times\BC^*$-fixed points in $\CQ_{\ul{d}}$ is
finite; it is described in~\cite{fk},~2.11. Recall that this fixed point set
is in bijection with the set of collections $\{\widehat{\ul{d}}\}$ of the
following data: a) a permutation $\sigma\in S_n$; b) a matrix $(d^0_{ij})$
(resp. $d^\infty_{ij}),\ i\geq j$ of nonnegative integers such that for
$i\geq j\geq k$ we have $d^0_{kj}\geq d^0_{ij}$ (resp. $d^\infty_{kj}\geq
d^\infty_{ij}$) such that $d_i=\sum_{j=1}^i(d^0_{ij}+d^\infty_{ij})$.
Abusing notation we
denote by $\widehat{\ul{d}}$ the corresponding
$\widetilde{T}\times\BC^*$-fixed point in $\CQ_{\ul{d}}$:

$\CW_1=\CO_\bC(-d^0_{11}\cdot0-d^\infty_{11}\cdot\infty)w_{\sigma(1)},$

$\CW_2=\CO_\bC(-d^0_{21}\cdot0-d^\infty_{21}\cdot\infty)w_{\sigma(1)}\oplus
\CO_\bC(-d^0_{22}\cdot0-d^\infty_{22}\cdot\infty)w_{\sigma(2)},$

$\ldots\ \ldots\ \ldots\ ,$

$\CW_{n-1}=\CO_\bC(-d^0_{n-1,1}\cdot0-d^\infty_{n-1,1}\cdot\infty)
w_{\sigma(1)}\oplus
\CO_\bC(-d^0_{n-1,2}\cdot0-d^\infty_{n-1,2}\cdot\infty)w_{\sigma(2)}\oplus$

$\oplus\ldots\oplus\CO_\bC(-d^0_{n-1,n-1}\cdot0-d^\infty_{n-1,n-1}\cdot\infty)
w_{\sigma(n-1)}.$

Also, we will write $\widehat{\ul{d}}=(\sigma,\widetilde{\ul{d}}{}^0,
\widetilde{\ul{d}}{}^\infty)$.

The localized equivariant cohomology $A$ is equipped with the basis
of direct images of the fundamental classes of the fixed points; by an abuse
of notation, this basis will be denoted $\{[\widehat{\ul{d}}]\}$.

\subsection{Chern classes of tautological bundles}
As in the proof of Theorem~\ref{braverman} and Corollary~\ref{mor} we
see that the $H^\bullet_{\widetilde{T}\times\BC^*}(pt)$-algebra
$H^\bullet_{\widetilde{T}\times\BC^*}(\CQ_{\ul{d}})$ is generated
by the K\"unneth components of the Chern classes $c_j^{(j)}(\ul{\CW}{}_i^\bC),\
c_j^{(j-1)}(\ul{\CW}{}_i^\bC),\ 1\leq j\leq i\leq n-1$, of the tautological
vector bundles $\ul{\CW}{}_i^\bC$ on $\CQ_{\ul{d}}\times\bC$.
We compute the operators
of multiplication by $c_j^{(j)}(\ul{\CW}{}_i^\bC),\
c_j^{(j-1)}(\ul{\CW}{}_i^\bC),\ 1\leq j\leq i\leq n-1$, in the equivariant
cohomology ring $H^\bullet_{\widetilde{T}\times\BC^*}(\CQ_{\ul{d}})$ in the
fixed point basis $\{[\widehat{\ul{d}}]\}$.

We introduce the following notation. For a function $f(x_1,\ldots,x_n,\hbar)\in
\BC(\ft\oplus\BC)$ and a permutation $\sigma\in S_n$ we set
$f^\sigma(x_1,\ldots,x_n,\hbar):=f(x_{\sigma(1)},\ldots,x_{\sigma(n)},\hbar)$.
Also, we set $\bar{f}(x_1,\ldots,x_n,\hbar):=f(x_1,\ldots,x_n,-\hbar)$.
For $1\leq j\leq i$, let $e_{ji}^0(\widehat{\ul{d}})$
(resp. $e_{ji}^\infty(\widehat{\ul{d}})$) stand for the sum of products
of the $j$-tuples of distinct elements of the set $\{-x_1+d_{i,1}^0\hbar,
\ldots,-x_i+d_{i,i}^0\hbar\}$ (resp. of the set $\{-x_1+d_{i,1}^\infty\hbar,
\ldots,-x_i+d_{i,i}^\infty\hbar\}$).

\begin{lem}
\label{trop}
The operator of multiplication by
$c_j^{(j)}(\ul{\CW}{}_i^\bC)$ (resp. by
$c_j^{(j-1)}(\ul{\CW}{}_i^\bC),\ 1\leq j\leq i\leq n-1$) is diagonal in
the basis $\{[\widehat{\ul{d}}]=[(\sigma,\widetilde{\ul{d}}{}^0,
\widetilde{\ul{d}}{}^\infty)]\}$
with eigenvalues $\{\frac{1}{2}(e_{ji}^0(\widehat{\ul{d}})+
e_{ji}^\infty(\widehat{\ul{d}}))^\sigma\}$ (resp.
$\{\frac{\hbar^{-1}}{2}(e_{ji}^\infty(\widehat{\ul{d}})-
e_{ji}^0(\widehat{\ul{d}}))^\sigma\}$).
\end{lem}

\begin{proof}
The same argument as in the proof of Theorem~\ref{braverman}.
\end{proof}

\begin{cor}
\label{c1}
The operator of multiplication by
$c_1^{(1)}(\ul{\CW}{}_i^\bC)$ is diagonal in the basis
$\{[\widehat{\ul{d}}]=[(\sigma,\widetilde{\ul{d}}{}^0,
\widetilde{\ul{d}}{}^\infty)]\}$ with eigenvalues
$-(x_1+\ldots+x_i)^\sigma+d_i\hbar$.
\end{cor}

\subsection{The double Universal Enveloping Algebra}
We denote by $\fU^2$ the universal enveloping algebra of $\fgl_n\oplus\fgl_n$
over the field $\BC(\ft\oplus\BC)$. For $1\leq i,j\leq n$ we denote
by $E_{ij}^{(1)}$ (resp. $E_{ij}^{(2)}$) the element
$(E_{ij},0)\in\fgl_n\oplus\fgl_n\subset\fU^2$ (resp. the element
$(0,E_{ij})\in\fgl_n\oplus\fgl_n\subset\fU^2$). We denote by
$\fU^2_{\leq0}$ the subalgebra of $\fU^2$ generated by
$E_{ii}^{(1)},E_{ii}^{(2)},\ 1\leq i\leq n,\ E_{i,i+1}^{(1)},E_{i,i+1}^{(2)},\
1\leq i\leq n-1$.
It acts on the field $\BC(\ft\oplus\BC)$ as follows:
$E_{i,i+1}^{(1)},E_{i,i+1}^{(2)}$ act trivially for any $1\leq i\leq n-1$, and
$E_{ii}^{(1)}$ (resp. $E_{ii}^{(2)}$) acts by multiplication by
$\hbar^{-1}x_i+i-1$ (resp. $-\hbar^{-1}x_i+i-1$).
We define the {\em universal Verma module}
$\fB$ over $\fU^2$ as $\fU^2\otimes_{\fU^2_{\leq0}}\BC(\ft\oplus\BC)$.
The universal Verma module $\fB$ is an irreducible $\fU^2$-module.

For a permutation $\sigma=(\sigma(1),\ldots,\sigma(n))\in S_n$ (the Weyl
group of $G=GL_n$) we consider a new action of $\fU^2_{\leq0}$ on
$\BC(\ft\oplus\BC)$ defined as follows:
$E_{i,i+1}^{(1)},E_{i,i+1}^{(2)}$ act trivially for any $1\leq i\leq n-1$, and
$E_{ii}^{(1)}$ (resp. $E_{ii}^{(2)}$) acts by multiplication by
$\hbar^{-1}x_{\sigma(i)}+i-1$ (resp.
$-\hbar^{-1}x_{\sigma(i)}+i-1$).
We define a module
$\fB^\sigma$ over $\fU^2$ as $\fU^2\otimes_{\fU^2_{\leq0}}\BC(\ft\oplus\BC)$
(with respect to the new action of $\fU^2_{\leq0}$ on $\BC(\f\oplus\BC)$).
Finally, we define $\fA:=\bigoplus_{\sigma\in S_n}\fB^\sigma$.

\subsection{The action of the double Universal Enveloping Algebra}
\label{operatory} The grading and the correspondences
$\CE_{\ul{d},i},\CE_{\ul{d},i}^0,\CE_{\ul{d},i}^\infty$
give rise to the following
operators on $A,B$:

$\ff_i^{(1)}=E_{i,i+1}^{(1)}=\hbar^{-1}\bp_*^0\bq^{0*}:\
A_{\ul{d}}\to A_{\ul{d}-i}\ \on{and}\
B_{\ul{d}}\to B_{\ul{d}-i}$;

$\ff_i^{(2)}=E_{i,i+1}^{(2)}=-\hbar^{-1}\bp_*^\infty\bq^{\infty*}:\
A_{\ul{d}}\to A_{\ul{d}-i}\ \on{and}\
B_{\ul{d}}\to B_{\ul{d}-i}$;

$\fe_i^{(1)}=E_{i+1,i}^{(1)}=-\hbar^{-1}\bq_*^0\bp^{0*}:\
A_{\ul{d}}\to A_{\ul{d}+i}\ \on{and}\
B_{\ul{d}}\to B_{\ul{d}+i}$;

$\fe_i^{(2)}=E_{i+1,i}^{(2)}=\hbar^{-1}\bq_*^\infty\bp^{\infty*}:\
A_{\ul{d}}\to A_{\ul{d}+i}\ \on{and}\
B_{\ul{d}}\to B_{\ul{d}+i}$;

$\ff_i^{\Delta}=\bp_*^\bC\bq^{\bC*}:\
A_{\ul{d}}\to A_{\ul{d}-i}\ \on{and}\
B_{\ul{d}}\to B_{\ul{d}-i}$;

$\fe_i^{\Delta}=-\bq_*^\bC\bp^{\bC*}:\
A_{\ul{d}}\to A_{\ul{d}+i}\ \on{and}\
B_{\ul{d}}\to B_{\ul{d}+i}$.

Furthermore, we define the action of $\BC[\ft\oplus\BC]$ on $B$
as follows: for $1\leq i\leq n-1,\ x_i$ acts on $B_{\ul{d}}$ by
multiplication by
$c_1^{(1)}(\ul{\CW}{}_{i-1}^\bC)-c_1^{(1)}(\ul{\CW}{}_i^\bC)-
(d_i-d_{i-1})\hbar$ (cf. Corollary~\ref{c1});
and $x_n$ acts by multiplication by
$e_1-x_1-\ldots-x_{n-1}$ (recall that $e_1$ is the generator of $H^2_G(pt)$).




\begin{thm}
\label{prav}
a) The operators $\fe_i^{(1)}=E_{i+1,i}^{(1)},\fe_i^{(2)}=E_{i+1,i}^{(2)},
\ff_i^{(1)}=E_{i,i+1}^{(1)},\ff_i^{(2)}=E_{i,i+1}^{(2)}$, along with the
action of $\BC(\ft\oplus\BC)$ on $B$ defined
in ~\ref{operatory} satisfy the relations in $\fU^2$, i.e.
they give rise to the action of $\fU^2$ on $B$;

b) There is a unique isomorphism $\Psi^2$ of $\fU^2$-modules
$B$ and $\fB$ carrying $1\in H^0_{\widetilde{T}\times\BC^*}(\CQ_0)\subset B$
to the lowest weight vector $1\in\BC(\ft\oplus\BC)\subset\fB$.
\end{thm}

\begin{proof}
We describe the matrix coefficients of
$\fe_i^{(1)}=E_{i+1,i}^{(1)},\fe_i^{(2)}=E_{i+1,i}^{(2)},
\ff_i^{(1)}=E_{i,i+1}^{(1)},\ff_i^{(2)}=E_{i,i+1}^{(2)}$ in the fixed point
basis $\{[\widehat{\ul{d}}]\}$.
Recall the matrix coefficients
$\fe_{i[\widetilde{\ul{d}},\widetilde{\ul{d}}{}']},
\ff_{i[\widetilde{\ul{d}},\widetilde{\ul{d}}{}']}$ computed in Proposition
~\ref{feigin'}. The proof of the following easy lemma is
omitted.

\begin{lem}
\label{bore}
Let $\widehat{\ul{d}}=(\sigma,\widetilde{\ul{d}}{}^0,
\widetilde{\ul{d}}{}^\infty)$, and
$\widehat{\ul{d}}{}'=(\sigma',\widetilde{\ul{d}}{}^0{}',
\widetilde{\ul{d}}{}^\infty{}')$. The matrix coefficients are as follows:
$\fe^{(1)}_{i[\widehat{\ul{d}},\widehat{\ul{d}}{}']}=
\delta_{\sigma,\sigma'}
(\fe_{i[\widetilde{\ul{d}}{}^0,\widetilde{\ul{d}}{}^0{}']})^\sigma;\
\fe^{(2)}_{i[\widehat{\ul{d}},\widehat{\ul{d}}{}']}=
\delta_{\sigma,\sigma'}
(\bar\fe_{i[\widetilde{\ul{d}}{}^\infty,
\widetilde{\ul{d}}{}^\infty{}']})^\sigma;\
\ff^{(1)}_{i[\widehat{\ul{d}},\widehat{\ul{d}}{}']}=
\delta_{\sigma,\sigma'}
(\ff_{i[\widetilde{\ul{d}}{}^0,\widetilde{\ul{d}}{}^0{}']})^\sigma;\
\ff^{(2)}_{i[\widehat{\ul{d}},\widehat{\ul{d}}{}']}=
\delta_{\sigma,\sigma'}
(\bar\ff_{i[\widetilde{\ul{d}}{}^\infty,
\widetilde{\ul{d}}{}^\infty{}']})^\sigma.
\qed$
\end{lem}

It follows that the operators
$\fe_i^{(1)}=E_{i+1,i}^{(1)},\fe_i^{(2)}=E_{i+1,i}^{(2)},
\ff_i^{(1)}=E_{i,i+1}^{(1)},\ff_i^{(2)}=E_{i,i+1}^{(2)}$
on $A$ defined
in ~\ref{operatory} satisfy the relations in $\fU^2$, i.e.
they give rise to the action of $\fU^2$ on $A$. Moreover, it follows that
the $\fU^2$-module $A$ is isomorphic to $\fA$. In order to describe the
image of the basis $\{[\widehat{\ul{d}}]\}$ under this isomorphism, we
introduce the following notation. First, we introduce a $\fU$-module
$\fV^\sigma:=\fU\otimes_{\fU_{\leq0}}\BC(\ft\oplus\BC)$ where $\fU_{\leq0}$
acts on $\BC(\ft\oplus\BC)$ as follows: $E_{i,i+1}$ acts trivially for any
$1\leq i\leq n-1$, and $E_{ii}$ acts by multiplication by
$\hbar^{-1}x_{\sigma(i)}+i-1$. Similarly, we introduce
a $\fU$-module $\bar\fV{}^\sigma:=\fU\otimes_{\fU_{\leq0}}\BC(\ft\oplus\BC)$
where $\fU_{\leq0}$
acts on $\BC(\ft\oplus\BC)$ as follows: $E_{i,i+1}$ acts trivially for any
$1\leq i\leq n-1$, and $E_{ii}$ acts by multiplication by
$-\hbar^{-1}x_{\sigma(i)}+i-1$.
Note that
$\fB^\sigma\simeq\fV^\sigma\otimes_{\BC(\ft\oplus\BC)}\bar\fV{}^\sigma$.

Now to a collection $\widetilde{\ul{d}}=(d_{ij})$, and a permutation
$\sigma\in S_n$, we associate a
Gelfand-Tsetlin pattern $\Lambda^\sigma=\Lambda^\sigma(\widetilde{\ul{d}}):=
(\lambda_{ij}^\sigma),\ n\geq i\geq j$, as follows:
$\lambda_{nj}^\sigma:=\hbar^{-1}x_{\sigma(j)}+j-1,\ n\geq j\geq 1;\
\lambda_{ij}^\sigma:=\hbar^{-1}x_{\sigma(j)}+j-1-d_{ij},\
n-1\geq i\geq j\geq1$. We define
$\xi^\sigma_{\widetilde{\ul{d}}}=\xi_{\Lambda^\sigma}\in\fV^\sigma$
by the formulas~(2.9)--(2.11) of~\cite{m} (where $\xi=\xi_0=1\in\fV^\sigma$).
Similarly, to a collection $\widetilde{\ul{d}}=(d_{ij})$, and a permutation
$\sigma\in S_n$, we associate a
Gelfand-Tsetlin pattern $\bar{\Lambda}{}^\sigma=
\bar{\Lambda}{}^\sigma(\widetilde{\ul{d}}):=
(\bar{\lambda}{}_{ij}^\sigma),\ n\geq i\geq j$, as follows:
$\bar{\lambda}{}_{nj}^\sigma:=
-\hbar^{-1}x_{\sigma(j)}+j-1,\ n\geq j\geq 1;\
\bar{\lambda}{}_{ij}^\sigma:=-\hbar^{-1}x_{\sigma(j)}+j-1-d_{ij},\
n-1\geq i\geq j\geq1$. We define
$\bar{\xi}{}^\sigma_{\widetilde{\ul{d}}}=
\xi_{\bar{\Lambda}{}^\sigma}\in\bar\fV{}^\sigma$
by the formulas~(2.9)--(2.11) of~\cite{m}
(where $\xi=\xi_0=1\in\bar\fV{}^\sigma$).
Finally, to a collection $\widehat{\ul{d}}=(\sigma,\widetilde{\ul{d}}{}^0,
\widetilde{\ul{d}}{}^\infty)$ we associate an element
$\xi_{\widehat{\ul{d}}}:=\xi^\sigma_{\widetilde{\ul{d}}{}^0}\otimes
\bar{\xi}{}^\sigma_{\widetilde{\ul{d}}{}^\infty}\in
\fV^\sigma\otimes_{\BC(\ft\oplus\BC)}\bar\fV{}^\sigma=\fB^\sigma$.

Theorem~\ref{feigin} and Lemma~\ref{bore}
implies that under the above isomorphism
$A\simeq\fA$ a basis element $[\widehat{\ul{d}}]$ goes to
$(-\hbar)^{|\ul{d}|}\xi_{\widehat{\ul{d}}}$.

We are ready to finish the proof of the theorem. Since the action of
$\widetilde{T}\times\BC^*$ on $\CQ_{\ul{d}}$ extends to the action of
$G\times\BC^*$, the equivariant cohomology
$H^\bullet_{\widetilde{T}\times\BC^*}(\CQ_{\ul{d}})$ is equipped with the
action of the Weyl group $S_n$, and
$H^\bullet_{G\times\BC^*}(\CQ_{\ul{d}})=
H^\bullet_{\widetilde{T}\times\BC^*}(\CQ_{\ul{d}})^{S_n}$.
It follows $B=A^{S_n}$. Since $B$ is closed with respect to the action of
the operators $\fe_i^{(1)}=E_{i+1,i}^{(1)},\fe_i^{(2)}=E_{i+1,i}^{(2)},
\ff_i^{(1)}=E_{i,i+1}^{(1)},\ff_i^{(2)}=E_{i,i+1}^{(2)}$
on $B$, part a) follows. Note however, that the action of $S_n$ on $A$
{\em does not} commute with the action of the above operators.

To prove part b), we describe the action of $S_n$ on $A$ explicitly in
the basis $\{[\widehat{\ul{d}}]\}$. For
$\widehat{\ul{d}}=(\sigma,\widetilde{\ul{d}}{}^0,
\widetilde{\ul{d}}{}^\infty),\ \sigma'\in S_n,\ f\in\BC(\ft\oplus\BC)$, we have
$\sigma'(f[\widehat{\ul{d}}])=
f^{\sigma'}[(\sigma'\sigma,\widetilde{\ul{d}}{}^0,
\widetilde{\ul{d}}{}^\infty)]$. We conclude that for
$\widehat{\ul{d}}=(1,\widetilde{\ul{d}}{}^0,
\widetilde{\ul{d}}{}^\infty)$ (so that $\xi_{\widehat{\ul{d}}}\in\fB^1=\fB$)
we have $(\Psi^2)^{-1}(f\xi_{\widehat{\ul{d}}})=\sum_{\sigma\in S_n}
f^\sigma[(\sigma,\widetilde{\ul{d}}{}^0,\widetilde{\ul{d}}{}^\infty)]$.
This completes the proof of the theorem.
\end{proof}

\begin{rem}
\label{compare}
{\em The operators $\ff_i^{\Delta},\fe_i^\Delta$ of~\ref{operatory}
were introduced in~\cite{fk}. It is easy to see that
$\ff_i^{\Delta}=\ff_i^{(1)}+\ff_i^{(2)},\
\fe_i^{\Delta}=\fe_i^{(1)}+\fe_i^{(2)}$.}
\end{rem}

\subsection{The double Gelfand-Tsetlin algebra} 
A completion of $\fB=\fV\otimes_{\BC(\ft\oplus\BC)}\fV$
contains the Whittaker vector $\sum_{\ul{d}}\fb_{\ul{d}}=\fb:=\fv\otimes\fv$.
It follows from the proof of Theorem~\ref{prav} that for the unit element
of the cohomology ring $1^{\ul{d}}\in H^0_{G\times\BC^*}(\CQ_{\ul{d}})$
we have $\Psi^2(1^{\ul{d}})=\fb_{\ul{d}}$. The double Gelfand-Tsetlin
subalgebra $\fG^2:=\fG\otimes\fG$ acts by endomorphisms of $\fB$. We denote
by $\fI^2_{\ul{d}}\subset\fG^2$ the annihilator ideal of the vector
$\fb_{\ul{d}}\in\fB$, and we denote by $\fG^2_{\ul{d}}$ the quotient of
$\fG^2$ by $\fI_{\ul{d}}^2$. The action of $\fG^2$ on $\fb_{\ul{d}}$ gives
rise to an embedding $\fG^2_{\ul{d}}\hookrightarrow\fB_{\ul{d}}$. The same
way as in Proposition~\ref{gc}.a) one proves that this embedding is an
isomorphism $\fG^2_{\ul{d}}\iso\fB_{\ul{d}}$.

\begin{prop}
\label{any}
The composite morphism $\Psi^2:\
H^\bullet_{G\times\BC^*}(\CQ_{\ul{d}})\otimes_{H^\bullet_{G\times\BC^*}(pt)}
\on{Frac}(H^\bullet_{G\times\BC^*}(pt))=B_{\ul{d}}\iso\fB_{\ul{d}}\iso
\fG^2_{\ul{d}}$ is an {\em algebra} isomorphism.
\end{prop}

\begin{proof}
As in the proof of Theorem~\ref{braverman} and Corollary~\ref{mor} we
see that the $H^\bullet_{G\times\BC^*}(pt)$-algebra
$H^\bullet_{G\times\BC^*}(\CQ_{\ul{d}})$ is generated
by the K\"unneth components of the Chern classes $c_j^{(j)}(\ul{\CW}{}_i^\bC),\
c_j^{(j-1)}(\ul{\CW}{}_i^\bC),\ 1\leq j\leq i\leq n-1$, of the tautological
vector bundles $\ul{\CW}{}_i^\bC$ on $\CQ_{\ul{d}}\times\bC$.
In order to prove the proposition, it suffices to check that the operators
of multiplication by $c_j^{(j)}(\ul{\CW}{}_i^\bC),\
c_j^{(j-1)}(\ul{\CW}{}_i^\bC),\ 1\leq j\leq i\leq n-1$, in the equivariant
cohomology ring $H^\bullet_{G\times\BC^*}(\CQ_{\ul{d}})$ lie in
$\fG^2_{\ul{d}}$. To this end we compute these operators explicitly in
the basis $\{(\Psi^2)^{-1}\xi_{\widehat{\ul{d}}},\
\widehat{\ul{d}}=(1,\widetilde{\ul{d}}{}^0, \widetilde{\ul{d}}{}^\infty)\}$
of $\fB$. Lemma~\ref{trop} implies that the operator of multiplication by
$c_j^{(j)}(\ul{\CW}{}_i^\bC)$ (resp. by
$c_j^{(j-1)}(\ul{\CW}{}_i^\bC),\ 1\leq j\leq i\leq n-1$) is diagonal in
the basis $\{(\Psi^2)^{-1}\xi_{\widehat{\ul{d}}},\
\widehat{\ul{d}}=(1,\widetilde{\ul{d}}{}^0, \widetilde{\ul{d}}{}^\infty)\}$
with eigenvalues $\{e_{ji}^0(\widehat{\ul{d}})+
e_{ji}^\infty(\widehat{\ul{d}})\}$ (resp.
$\{\hbar^{-1}(e_{ji}^\infty(\widehat{\ul{d}})-e_{ji}^0(\widehat{\ul{d}}))\}$).
As in the proof of Theorem~\ref{braverman} we see that
$e_{ji}^0(\widehat{\ul{d}})$
is the eigenvalue of the element of
$\fG_{\ul{d}}\otimes1\subset\fG^2_{\ul{d}}$, and
$e_{ji}^\infty(\widehat{\ul{d}})$ is the eigenvalue of the same element in
the other copy of the Gelfand-Tsetlin subalgebra
$1\otimes\fG_{\ul{d}}\subset\fG^2_{\ul{d}}$,
hence $c_j^{(j)}(\ul{\CW}{}_i^\bC)$ and $c_j^{(j-1)}(\ul{\CW}{}_i^\bC)$
lie in $\fG^2_{\ul{d}}$.

\end{proof}

\subsection{Integral forms}
\label{if}
Recall the notations of~\ref{kuzn}. We consider the correspondences
$\CE^0_{\ul{d},\alpha_{ij}}\subset\CQ_{\ul{d}}\times\CQ_{\ul{d}+\alpha_{ij}}$
(resp. $\CE^\infty_{\ul{d},\alpha_{ij}}\subset
\CQ_{\ul{d}}\times\CQ_{\ul{d}+\alpha_{ij}}$)
defined exactly as in {\em loc. cit.} (resp. replacing the condition
in~\ref{kuzn}.b
by the condition that $\CW_\bullet/\CW'_\bullet$ is supported at
$\infty\in\bC$). We denote by $\bp^0_{ij}:\ \CE^0_{\ul{d},\alpha_{ij}}\to
\CQ_{\ul{d}},\ \bq^0_{ij}:\ \CE^0_{\ul{d},\alpha_{ij}}\to
\CQ_{\ul{d}+\alpha_{ij}},\ \bp^\infty_{ij}:\ \CE^\infty_{\ul{d},\alpha_{ij}}\to
\CQ_{\ul{d}},\ \bq^\infty_{ij}:\ \CE^\infty_{\ul{d},\alpha_{ij}}\to
\CQ_{\ul{d}+\alpha_{ij}}$ the natural proper projections. We also consider
the correspondences and projections $\CE^\bC_{\ul{d},\alpha_{ij}}
\subset\CQ_{\ul{d}}\times\CQ_{\ul{d}+\alpha_{ij}},\ \bp^\bC_{ij},\
\bq^\bC_{ij}$ defined as above but without any restriction on the support
of $\CW_\bullet/\CW'_\bullet$.

We consider the following operators on $'B$:

$\ul{E}{}_{ij}^{(1)}=\bp_{ij*}^0\bq_{ij}^{0*}:\
'B_{\ul{d}}\to\ 'B_{\ul{d}-\alpha_{ij}}$;

$\ul{E}{}_{ij}^{(2)}=-\bp_{ij*}^\infty\bq_{ij}^{\infty*}:\
'B_{\ul{d}}\to\ 'B_{\ul{d}-\alpha_{ij}}$;

$\ul{E}{}_{ji}^{(1)}=(-1)^{i-j}\bq_{ij*}^0\bp_{ij}^{0*}:\
'B_{\ul{d}}\to\ 'B_{\ul{d}+\alpha_{ij}}$;

$\ul{E}{}_{ji}^{(2)}=-(-1)^{i-j}\bq_{ij*}^\infty\bp_{ij}^{\infty*}:\
'B_{\ul{d}}\to\ 'B_{\ul{d}+\alpha_{ij}}$;

$E_{ij}^{\Delta}=\bp_{ij*}^\bC\bq_{ij}^{\bC*}:\
'B_{\ul{d}}\to\ 'B_{\ul{d}-\alpha_{ij}}$;

$E_{ji}^{\Delta}=(-1)^{i-j}\bq_{ij*}^\bC\bp_{ij}^{\bC*}:\
'B_{\ul{d}}\to\ 'B_{\ul{d}+\alpha_{ij}}$.

We have $E_{ij}^{\Delta}=\hbar^{-1}(\ul{E}{}_{ij}^{(1)}+\ul{E}{}_{ij}^{(2)}),\
E_{ji}^{\Delta}=\hbar^{-1}(\ul{E}{}_{ji}^{(1)}+\ul{E}{}_{ji}^{(2)})$.

We define $\ul{\fU}^2\subset\fU^2$ as the $\BC[\ft\oplus\BC]$-subalgebra
generated by $\ul{y}^{(1)},\ \ul{y}^{(2)},\ y^\Delta,\ y\in\fgl_n$,
where $\ul{y}^{(1)}$ (resp. $\ul{y}^{(2)}$) stands for $\hbar(y,0)$
(resp. $\hbar(0,y)$), and $y^\Delta$ stands for $(y,y)$.
Then it is easy to see that the above operators give rise to the action
of $\ul{\fU}^2$ on $'B$.

Also, it is easy to check that $\ul{\fU}^2/(\hbar=0)$ is isomorphic to
the algebra ${\mathsf U}:=(\BC[\fgl_n]\rtimes U(\fgl_n))\otimes\BC[\ft]$
(the semidirect product with respect to the adjoint action).
Hence $\bar{B}:=\ 'B/(\hbar=0)=\oplus_{\ul{d}}H^\bullet_G(\CQ_{\ul{d}})$
inherits an action of ${\mathsf U}$.

\begin{conj}
\label{ancient}
The ${\mathsf U}$-module $\bar B$ is isomorphic to
$H^{\frac{n(n-1)}{2}}_{\{\fg_{\leq0}\}}(\fgl_n,\CO)$. Under this
isomorphism, the action of
$H^\bullet_G(pt)$ on $\bar B$ corresponds to the action of
$\BC[\fgl_n]^G$ on $H^{\frac{n(n-1)}{2}}_{\{\fg_{\leq0}\}}(\fgl_n,\CO)$.
\end{conj}

Conjecture~6.4 of~\cite{fk} on the direct sum of {\em nonequivariant}
cohomology of $\CQ_{\ul{d}}$ is an immediate corollary of~\ref{ancient}.

\subsection{Relative Laumon spaces}
\label{relat}
We propose a generalization of Conjecture 6.4 of~\cite{fk} in a different
direction.
Let $\ul{d}=(d_1,\ldots,d_n)$ be an $n$-tuple of integers (not
necessarily positive).
Let $\sQ_{\ul{d}}$ be the moduli stack of flags of locally free sheaves
$\CW_1\subset\ldots\subset\CW_{n-1}\subset\CW_n$
on $\bC$ such that $\on{rk}(\CW_i)=i,\
\deg(\CW_i)=d_i$ (see~\cite{la2}). We have a representable
projective
morphism $\pi:\ \sQ_{\ul{d}}\to\Bun,\ \CW_\bullet\mapsto\CW_n$, where $\Bun$
stands for the moduli stack of $GL_n$-bundles on $\bC$. The fiber of $\pi$
over the trivial $GL_n$-bundle (an open point of $\Bun$) is $\CQ_{\ul{d}}$.
The correspondences
$\CE_{\ul{d},\alpha_{ij}}^0,\CE_{\ul{d},\alpha_{ij}}^\infty,
\CE_{\ul{d},\alpha_{ij}}^\bC$
of subsection~\ref{if} make perfect sense for the
stacks $\sQ_{\ul{d}}$ in place of $\CQ_{\ul{d}}$. As in {\em loc. cit.},
they give rise to the operators
$\ul{E}{}_{ij}^{(1)},\ul{E}{}_{ij}^{(2)},E_{ij}^\Delta$, etc. on the
constructible complex
$\sB:=\bigoplus_{\ul{d}\in\BZ^{n-1}}\pi_*\ul\BC{}_{\ul{d}}$
on $\Bun$ (where $\ul\BC{}_{\ul{d}}$ stands for the constant sheaf on
$\sQ_{\ul{d}}$). This constructible complex is the geometric Eisenstein
series of~\cite{la2}. The above operators give rise to the action of
$\bar\sU:=\BC[\fgl_n]\rtimes U(\fgl_n)$ on $\sB$: this follows from the
results of~\ref{if} by the argument of~\cite{bf},~3.8--3.11.
In particular, $\bar\sU$ acts on the stalks of $\sB$, and we propose a
conjecture describing the resulting $\bar\sU$-modules.

Recall that the isomorphism classes of $GL_n$-bundles on $\bC$ are
parametrized by the set $X^+$ of dominant weights of $GL_n$.
For $\eta\in X^+$ we denote by $\sB_\eta$ the corresponding stalk of $\sB$.
Also, we will denote by $\CO(\eta)$ the corresponding line bundle on the flag
variety $\CB_n$ of $GL_n$. We will keep the same notation for the lift
of $\CO(\eta)$ to the cotangent bundle $T^*\CB_n$. We will denote by
$L_\eta$ the direct image of $\CO(\eta)$ under the Springer resolution
morphism $T^*\CB_n\to\CN$ to the nilpotent cone $\CN\subset\fgl_n$.
The cohomology of the coherent sheaf $L_\eta$ carries a natural action
of $\bar\sU$.

\begin{conj}
\label{fren}
The $\bar\sU$-module $\sB_\eta$ is isomorphic to
$H^{\frac{n(n-1)}{2}}_{\fg_{<0}}(\CN,L_\eta)$ (cf. Conjecture~7.8
of~\cite{ffkm}).
\end{conj}

\section{Equivariant $K$-ring of $\fQ_{\ul{d}}$ and quantum Gelfand-Tsetlin
algebra}

\subsection{Quantum Universal Enveloping Algebra}
\label{quantum}
We preserve the setup of~\cite{bf} with the following
slight changes of notation.
Now $\widetilde{T}$ stands for a $2^n$-cover of a Cartan torus of $GL_n$
as opposed to $SL_n$ in {\em loc. cit.}
Now $U'$ stands for the quantum universal enveloping algebra of $\fgl_n$
over the field $\BC(\widetilde{T}\times\BC^*)$, as opposed to the quantum
universal enveloping algebra of $\fsl_n$ in 2.26 of {\em loc. cit.}

For the quantum universal enveloping algebra of $\fgl_n$ we follow the
notations of section~2 of~\cite{mtz}. Namely, $U'$ has generators
$t_{ij},\bar{t}_{ij},\ 1\leq i,j\leq n$, subject to relations~(2.4) of
{\em loc. cit.} The standard Chevalley generators are expressed via
$t_{ij},\bar{t}_{ij}$ as follows:
$$K_i=t_{i+1,i+1}\bar{t}_{ii},\ E_i=(v-v^{-1})^{-1}\bar{t}_{ii}t_{i+1,i},\
F_i=-(v-v^{-1})^{-1}\bar{t}_{i,i+1}t_{ii}$$
(note that this presentation differs from the one in~(2.6) of {\em loc. cit.}
by an application of Chevalley involution).
Note also that $U'$ is generated by $t_{ii},\bar{t}_{ii},\ 1\leq i\leq n;\
t_{i+1,i},\bar{t}_{i,i+1},\ 1\leq i\leq n-1$.
We denote by $U'_{\leq0}$ the subalgebra of $U'$ generated by
$t_{ii},\bar{t}_{ii},\bar{t}_{i,i+1}$.
It acts on the field $\BC(\widetilde{T}\times\BC^*)$ as follows:
$\bar{t}_{i,i+1}$ acts trivially for any $1\leq i\leq n-1$, and
$\bar{t}_{ii}=t_{ii}^{-1}$ acts by multiplication by
$t_i^{-1}v^{1-i}$. We define the {\em universal Verma module}
$\fM$ over $U'$ as
$\fM:=U'\otimes_{U'_{\leq0}}\BC(\widetilde{T}\times\BC^*)$.

Recall that $M=\oplus_{\ul{d}}M_{\ul{d}},\
M_{\ul{d}}=K^{\widetilde{T}\times\BC^*}(\fQ_{\ul{d}})
\otimes_{\BC[\widetilde{T}\times\BC^*]}\BC(\widetilde{T}\times\BC^*)$
(cf.~\cite{bf},~2.7).

We define the following operators on $M$ (well defined since the
correspondences $\fE_{\ul{d},i}$ are smooth, and the
$\widetilde{T}\times\BC^*$-fixed point sets are finite):
$$t_{ii}=t_iv^{d_{i-1}-d_i+i-1}:\ M_{\ul{d}}\to M_{\ul{d}};\
\bar{t}_{ii}=t_{ii}^{-1};$$
$$\bar{t}_{i,i+1}=(v^{-1}-v)t_{i+1}^it_i^{-i-1}v^{(2i+1)d_i-(i+1)d_{i-1}
-id_{i+1}-2i+1}\bp_*\bq^*:\ M_{\ul{d}}\to M_{\ul{d}-i};$$
$$t_{i+1,i}=(v^{-1}-v)t_{i+1}^{-i-1}t_i^iv^{id_{i-1}+(i+1)d_{i+1}-(2i+1)d_i
-1}\bq_*(\fL_i\otimes\bp^*):\ M_{\ul{d}}\to M_{\ul{d}+i}.$$
According to Theorem~2.12 of~\cite{bf}, these operators satisfy the
relations in $U'$, i.e. they give rise to the action of $U'$ on $M$.
Moreover, there is a unique isomorphism $\psi:\ M\to\fM$ carrying
$[\CO_{\fQ_0}]\in M$ to the lowest weight vector
$1\in\BC(\widetilde{T}\times\BC^*)\subset\fM$.

\subsection{Quantum Gelfand-Tsetlin basis}
The construction of Gelfand-Tsetlin basis for the representations of
quantum $\fgl_n$ goes back to M.~Jimbo~\cite{j}. We will follow the approach
of~\cite{mtz}. Given $\widetilde{\ul{d}}$ and the corresponding Gelfand-Tsetlin
pattern $\Lambda=\Lambda(\widetilde{\ul{d}})$ (see subsection~\ref{classical'}),
we define $\xi_{\widetilde{\ul{d}}}=\xi_\Lambda\in\fM$ by the formula~(5.12)
of~\cite{mtz}. According to Proposition~5.1 of {\em loc. cit.}, the set
$\{\xi_{\widetilde{\ul{d}}}\}$ (over all collections $\widetilde{\ul{d}}$)
forms a basis of $\fM$.

Recall the basis $\{[\widetilde{\ul{d}}]\}$ of $M$ introduced in~2.16
of~\cite{bf}: the direct images of the structure sheaves of the torus-fixed
points. The following theorem is proved
absolutely similarly to Theorem~\ref{feigin}, using Proposition~5.1
of~\cite{mtz}.

\begin{thm}
\label{Feigin}
The isomorphism $\psi:\ M\to\fM$ of subsection~\ref{quantum} takes
$[\widetilde{\ul{d}}]$ to $c_{\widetilde{\ul{d}}}\xi_{\widetilde{\ul{d}}}$
where
$$c_{\widetilde{\ul{d}}}=(v^2-1)^{-|\ul{d}|}
v^{|\ul{d}|+\sum_{i}{id_{i-1}d_i}-\sum_{i}
{\frac{2i+1}{2}d_i^2}-\frac{1}{2}\sum_{i,j}{d_{i,j}^2}}
\prod_{i}{t_i^{i(d_i-d_{i-1})}}
\prod_{j}{t_j^{\sum_{k\geq j}{d_{k,j}}}}.$$
\qed
\end{thm}

\subsection{Quantum Casimirs}
Let $Cas^v_k$ be the quantum Casimir element of the completion of the
quantum universal enveloping algebra $U_v(\fgl_{k})$. The quantum Casimir
element is defined in section~6.1. of \cite{Lu} in terms of the universal
$R$-matrix lying in the completion of $U_v(\fgl_k)\otimes U_v(\fgl_k)$.
According to \cite{Lu}, Proposition~6.1.7, the eigenvalue of $Cas^v_k$ on
the Verma module over $U_v(\fgl_{k})$ with the highest weight $\lambda_k$
is $v^{-(\lambda_k,\lambda_k+2\rho_k)}$. This means that the operator
$Cas^v_k$ is diagonal in the Gelfand-Tsetlin basis, and the eigenvalue of
$Cas^v_k$ on the basis vector $\xi_{\widetilde{\ul{d}}}=\xi_\Lambda$ is
$v^{-\sum_{j\leq k}\lambda_{kj}(\lambda_{kj}+k-2j+1)}$
(with $t_iv^{j-1-d_{ki}}=v^{\lambda_{ki}}$).
Consider the following ``corrected'' Casimir operators $$
\wt{Cas^v_k}:=Cas^v_k\cdot \prod\limits_{j=1}^k t_{jj}^{k-2}
v^{\sum\limits_{j=1}^k (\lambda_{nj}-j)
(\lambda_{nj}-j+1)-
\frac{k(k-1)(k-2)}{3}}
$$

Lemma~\ref{ias} admits the following

\begin{cor}
\label{Princeton}
a) The operator of tensor multiplication by the class $[\D_k]$ in $M$
is diagonal in the basis $\{[\widetilde{\ul{d}}]\}$, and the eigenvalue
corresponding to $\widetilde{\ul{d}}=(d_{ij})$ equals
$\prod_{j\leq k}t_j^{2-2d_{kj}}v^{d_{kj}(d_{kj}-1)}$.

b) The isomorphism $\psi:\ M\to\fM$ carries the operator of tensor
multiplication by $[\D_k]$ to the operator $\wt{Cas^v_k}^{-\frac{1}{2}}$.
\end{cor}

\begin{proof}
a) is immediate.

b) straightforward from a) and the formula for eigenvalues of $Cas^v_k$.
\end{proof}

\subsection{Quantum Gelfand-Tsetlin algebra and $K$-rings}
Recall the Whittaker vector $\fk=\sum_{\ul{d}}\fk_{\ul{d}}$ of~\cite{bf}~2.30.
According to Proposition~2.31 of {\em loc. cit.},
$\psi[\CO_{\ul{d}}]=\fk_{\ul{d}}$.

Consider the ``quantum Gelfand-Tsetlin algebra''
$\CG\subset\operatorname{End}(\fM)$
generated by all the $\wt{Cas^v_k}^{-\frac{1}{2}}$ over the field
$\BC(\widetilde{T}\times\BC^*)$.
We denote by $\CI_{\ul{d}}\subset\CG$ the annihilator ideal of the vector
$\fk_{\ul{d}}\in\fV$, and we denote by $\CG_{\ul{d}}$ the quotient algebra
of $\fG$ by $\CI_{\ul{d}}$. The action of $\CG$ on $\fk_{\ul{d}}$ gives rise
to an embedding $\CG_{\ul{d}}\hookrightarrow\fM_{\ul{d}}$.

\begin{prop}
\label{GC}
a) $\CG_{\ul{d}}\iso\fM_{\ul{d}}$.

b) The composite morphism
$\psi:\ K^{\widetilde{T}\times\BC^*}(\fQ_{\ul{d}})
\otimes_{\BC[\widetilde{T}\times\BC^*]}\BC(\widetilde{T}\times\BC^*)=
M_{\ul{d}}\iso\fM_{\ul{d}}\iso\CG_{\ul{d}}$
is an {\em algebra} isomorphism.

c) The algebra $K^{\widetilde{T}\times\BC^*}(\fQ_{\ul{d}})
\otimes_{\BC[\widetilde{T}\times\BC^*]}\BC(\widetilde{T}\times\BC^*)$
is generated by
$\{[\D_k]:\ k\geq2,\ d_k\ne0\ne d_{k-1}\}$.
\end{prop}

\begin{proof}
The proof is the same as for Proposition~\ref{gc}.
\end{proof}

\end{document}